\documentclass{amsart}
\usepackage{amsmath}
\usepackage{amssymb}
\usepackage{amsthm}
\usepackage{amsfonts}
\usepackage{mathtools}
\usepackage{fullpage}
\usepackage{booktabs}
\usepackage{enumerate}
\usepackage{graphicx}
\usepackage{xcolor}
\usepackage[shortlabels]{enumitem}
\usepackage{esint}
\usepackage{tikz-cd}
\usepackage{float}
\usepackage{stmaryrd}
\usepackage{hyperref}
\usepackage{url}
\usepackage{fontawesome}
\usepackage{biblatex}
\usepackage{faktor}
\usepackage{leftindex}
\usepackage{MnSymbol}
\usepackage{dashbox}

\newtheorem{theorem}{Theorem}

\newtheorem{corollary}{Corollary}
\newtheorem{proposition}{Proposition}
\newtheorem{remark}{Remark}
\newtheorem*{theorem*}{Theorem}
\newtheorem{lemma}{Lemma}
\newtheorem*{corollary*}{Corollary}
\newtheorem*{remark*}{Remark}
\newtheorem*{task*}{Task}
\newtheorem*{wish*}{Wish}
\newtheorem*{notation*}{Notation}

\theoremstyle{definition}
\newtheorem{example}{Example}
\newtheorem{definition}{Definition}
\newtheorem*{solution*}{Solution}
\newtheorem*{idea*}{Idea}

\newcommand{\Z}{\mathbb{Z}}

\newcommand{\C}{\mathbb{C}}

\newcommand{\st}{\ \big|\ }

\renewcommand{\O}{\mathcal{O}}
\newcommand{\B}{\mathcal{B}}

\newcommand{\Id}{\operatorname{Id}}

\newcommand{\dee}{\partial}

\renewcommand{\bar}{\overline}

\newcommand{\im}{\operatorname{im}}
\newcommand{\vspan}{\operatorname{span}}

\newcommand{\M}{\mathcal{M}}
\renewcommand{\P}{\mathbb{P}}

\renewcommand{\hat}{\widehat}

\newcommand{\Hom}{\operatorname{Hom}}

\renewcommand{\tilde}{\widetilde}

\newcommand{\res}{\operatorname{res}}

\newcommand{\Spec}{\operatorname{Spec}}

\newcommand{\m}{\mathfrak{m}}
\newcommand{\Aut}{\operatorname{Aut}}
\newcommand{\A}{\mathbb{A}}

\newcommand{\Gal}{\operatorname{Gal}}

\newcommand{\SL}{\operatorname{SL}}
\newcommand{\GL}{\operatorname{GL}}

\renewcommand{\sl}{\mathfrak{sl}}

\newcommand{\D}{\mathcal{D}}
\newcommand{\g}{\mathfrak{g}}
\newcommand{\h}{\mathfrak{h}}

\renewcommand{\epsilon}{\varepsilon}

\newcommand{\stab}{\operatorname{stab}}
\newcommand{\codim}{\operatorname{codim}}
\newcommand{\mSpec}{\operatorname{mSpec}}

\newcommand{\fc}{\mathfrak{c}}
\newcommand{\Frac}{\operatorname{Frac}}
\newcommand{\PC}{\operatorname{PC}}
\newcommand{\V}{\mathcal{V}}
\newcommand{\dashto}{\dashedrightarrow}

\newcommand\snode[2]{\fbox{\begin{tabular}{@{}c@{}}\small #1 \\ \small #2\end{tabular}}}
\newcommand\dnode[2]{\dbox{\begin{tabular}{@{}c@{}}\small #1 \\ \small #2\end{tabular}}}

\bibliography{refs.bib}

\title{Restriction theorems: from orbits and Chevalley to periods and Galois}

\dedicatory{In honor of Professor Shing-Tung Yau on the occasion of his seventy-fifth birthday}

\author{Bong Lian}
\address{School of Mathematical Sciences, Fudan University, Shanghai, 200433, China}
\address{Shanghai Institute for Mathematics and Interdisciplinary Sciences (SIMIS), Shanghai, 200433, China}
\email{lianbong@gmail.com}

\author{Kamryn Spinelli}
\address{Department of Mathematics and Statistics, Queen's University, Kingston, ON K7L 3N6, Canada}
\email{k.spinelli@queensu.ca}

\begin{document}

\begin{abstract}
    Using a new approach based on Galois theory, we study subvarieties of complex representations of reductive groups which satisfy restriction properties on their invariant rings and function fields, along the lines of the Chevalley restriction theorem. For a certain well-behaved class of representations, we explicitly parametrize candidates for these restriction properties and explain a technique to understand their deformations in complex families. We also give algebraic and geometric characterizations of the Chevalley restriction property which clarify how this perspective connects back to previous orbit-theoretic approaches. Finally, we utilize these restriction properties to prove explicit formulas for period integrals of some Calabi-Yau families. The key insight is that the restriction property on function fields can be leveraged to locally interpolate between the algebraic and analytic settings. Using this technique, we lift hypergeometric period formulas from subfamilies to obtain novel explicit formulas for periods of Calabi-Yau double covers of projective spaces and elliptic curves in $\P^2$, expressed in terms of invariant functions on their parameter spaces.
\end{abstract}

\maketitle

\section{Introduction} \label{sec:introduction}

Understanding the space of orbits of representation $V$ of a reductive group $G$ as an algebraic variety is a notoriously subtle task \cite{Mumford1977}. The natural idea of setting $V \sslash G = \mSpec (\C[V]^G)$, or in other words defining the ring of functions on the quotient to agree with the ring of $G$-invariant functions on $V$, leads only to a rough parametrization of the so-called semistable orbits whose closures do not contain the origin, and throws away lots of information about the unstable orbits. Furthermore, the fibers of the canonical morphism of algebraic varieties $V \to V \sslash G$ are only truly orbits of the $G$-action above the stable locus of orbits which are closed and whose points have finite stabilizers. Algebraically complicating the situation is that understanding the invariant rings themselves is generally not straightforward, and highly involved computations may be necessary to determine the degrees of the generators of the invariant rings, let alone the explicit forms of these generators themselves. For example, the invariant ring of the ternary cubic is generated as a $\C$-algebra by the Aronhold invariants $S$ and $T$ whose explicit forms together occupy about a page and a half in \cite{sturmfels}, and the invariant ring of the ternary quartic is generated by thirteen elements, seven of which are algebraically independent, and which have degrees ranging from 3 to 27 \cite{DIXMIER1987279,shioda-ternary-quartic}.

On the other hand, if one instead takes a birational perspective, there is always a geometric quotient. Namely, Rosenlicht showed that for any action of an algebraic group $G$ on a variety $X$, there is a variety $Z$ whose function field is the field $\C(X)^G$ of $G$-invariant rational functions, and such that the induced dominant rational map $\tau : X \dashto Z$ has $G$-orbits as its fibers \cite{rosenlicht}. Roughly, Rosenlicht's theorem says that the invariant field is a finer algebraic object for parametrizing $G$-orbits, up to a Zariski-small subset of such orbits; the algebraic cost you pay is that the invariant field $\C(X)^G$ contains the invariant ring $\C[X]^G$ and perhaps other rational functions depending on the character theory of $G$. Rosenlicht also proved that whenever $G$ is a solvable algebraic group, there is a rational map $\sigma : Z \dashto X$ so that $\tau \circ \sigma$ is the identity; in other words, the image of $\sigma$ meets most orbits exactly once. Since most reductive groups occurring in nature are not solvable, this fact is not particularly suited for our case.

When $G$ is reductive, however, Luna gave a similar result with the tradeoff that the stabilizers of points enter the picture. This is known as Luna's slice theorem. Given a point $x$ of an affine $G$-variety $X$ with closed stabilizer $\stab(x)$, he showed that there is an affine subvariety $S$ containing $x$ which is stable under $\stab(x)$, a Zariski open set $U \subseteq V$, and a natural \'etale map $S \times_{\stab(x)} G \to U$. This result has the same spirit as Palais's slice theorem in the smooth category, which says that for any point $x$ in a smooth $G$-space $M$, there is a submanifold $S$ containing $x$ for which $M$ is locally diffeomorphic to $S \times (G / \stab(x))$ near $x$ \cite{palais-slice}.

Many results in invariant theory have been motivated essentially by the idea of finding a global slice for the $G$-action, or something sufficiently close to it. (In fact, Luna's slice theorem plays a crucial role in the proofs of some of the results described below.) In practice, the idea has historically been to find a vector subspace $Y \subset V$ and a finite group $W$ acting linearly on $Y$ such that the restriction of polynomials $\res_{V \to Y}$ sends $\C[V]^G$ isomorphically to $\C[Y]^W$; then this reduces the question of determining the invariant ring under the reductive group action to the simpler question of determining the invariant ring under a finite group action. In this article, we will call a statement which guarantees the existence of such a $Y$ and $W$ a \textit{restriction theorem}.

The first, and probably best-known, example of such a theorem is Chevalley's restriction theorem \cite{humphreys}. It deals with the specific case of the adjoint representation on a semisimple Lie algebra. Namely, let $G$ be a connected semisimple Lie group, $\g$ its Lie algebra, $\h$ a Cartan subalgebra, and $W$ its Weyl group, and let $G$ act on $\g$ by the adjoint action and $W$ on $\h$ in the usual way. Then the Chevalley restriction theorem states that $\res_{\g \to \h} : \C[\g]^G \to \C[\h]^W$ is an isomorphism of invariant rings.

For instance, take $G = \SL_2(\C)$. In this case $\h$ is one-dimensional and $W = \Z / 2\Z$ acts by reflection on $\h$. Denoting by $z_1, z_2, z_3$ the coordinate functions on $\g$ with respect to the usual basis $H, X, Y$ for $\sl_2(\C)$ and taking $\h = \vspan_\C(H)$, we have $\C[\h]^W = \C[z_1^2]$ and therefore $\C[\g]^G$ is generated by a single polynomial which restricts to $z_1^2$ on $\h$. By manual computations involving the $G$-action on $\g$, one can find that $\C[\g]^G$ is generated by $z_1^2 + z_2 z_3$, which is Killing dual to the Casimir element of $\sl_2(\C)$ and is readily verified to be $G$-invariant.

Chevalley's theorem caught the eye of many invariant theorists, who subsequently worked to generalize it to the context of reductive group actions. One such program was the study of polar representations by Dadok and Kac \cite{Dadok1985, DADOK1985504}. In their orbit-theoretic approach, they studied vector subspaces of a reductive group representation $V$ of the form
\[\fc_v = \{x \in V \st \g \cdot x \subseteq \g \cdot v\}\]
whose $G$-orbit through each point has a tangent space parallel to the $G$-orbit through a given $v \in V$. Dadok and Kac proved that when $V$ is polar, meaning that some such subspace has the same dimension as $V \sslash G$, it satisfies a restriction theorem. In this case $\fc_v$ is called a Cartan subspace of $V$ and there is a finite group $W$, called the Weyl group, which is defined via a normalizer-mod-stabilizer construction inside $G$ and acts on $\fc_v$, for which $\res_{V \to \fc_v} : \C[V]^G \to \C[\fc_v]^W$ is an isomorphism.

All of the adjoint representations addressed by the Chevalley restriction theorem are polar, and Dadok and Kac further enumerated all polar representations of simple reductive groups. A simple example would be the action of $G = \SL_2(\C)$ on $V = M_{2 \times 2}(\C)$, the space of $2 \times 2$ complex matrices. The subspace
\[\fc_{\Id} = \left\{ \begin{bmatrix}
    x & 0 \\
    0 & x
\end{bmatrix} {\ \Bigg |\ } x \in \C \right\}\]
is thus a Cartan subspace, and the Weyl group $W = \Z / 2 \Z$ acts on $\fc_{\Id}$ by inversion. It is well-known that in this case $\C[V]^G = \C[\Delta]$, where $\Delta$ denotes the determinant, and one can easily check that its restriction to $\fc_{\Id}$, $x^2$, generates $\C[\fc_{\Id}]^W$. We will see that the theory of the present article gives a new interpretation of Dadok and Kac's Weyl group as a Galois group of a suitable field extension.

Another restriction theorem generalizing the work of Chevalley is due to Luna and Richardson \cite{luna-richardson}. They showed that in fact any representation $V$ of a reductive group $G$ satisfies a weak kind of restriction theorem. Namely, for generic $x \in V$, there is a group $W$ defined by a normalizer-mod-stabilizer construction in $G$ which acts on the fixed point set $F = V^{\stab x}$ under the action of the stabilizer of $x$. In this case, $\res_{V \to F} : \C[V]^G \to \C[F]^W$ is an isomorphism.

The keen reader will notice that it's not guaranteed that $W$ is finite; even in simple examples, it is not uncommon for $W$ to be infinite if stabilizers of points in $V$ are generically small. The case of $\SL_2(\C)$ acting on $M_{2 \times 2}(\C)$ is an extreme example: any invertible matrix has trivial stabilizer, and so the output of the Luna-Richardson theorem is $F = M_{2 \times 2}(\C)$ and $W = \SL_2(\C)$. However, when the group $W$ is finite, the theorem can be used to deduce information on the invariant ring $\C[V]^G$, and Luna and Richardson used this technique to compute the degrees of the invariant polynomials of some high-dimensional representations of $\SL_7(\C)$, $\SL_8(\C)$, and $\operatorname{Spin}_{13}(\C)$.

Despite the obvious power of these restriction theorems, some important questions still remain unanswered. For one, taking a geometric rather than representation-theoretic point of view, there is no obvious reason to require the target subvariety of a restriction theorem to be a vector subspace (or even that the ambient $V$ must be a vector space). Furthermore, in each of the three restriction theorems above, the construction of the Weyl group is somewhat ad-hoc, and sorely lacks a satisfying uniform description. Our goals in this article are to clarify these points, and in particular we will explain that (1) these restriction theorems have a natural extension to arbitrary subvarieties of $V$, not just subspaces; (2) the notion of restriction on function fields is a more natural condition to study from an algebraic perspective, especially through the lens of Galois theory, which plays a crucial role in restriction phenomena and can be used to universally define the Weyl group; and (3) that these insights are of profound and previously unnoticed utility in mirror symmetry and the study of period integrals of Calabi-Yau families. The purpose of this article is to explain a new approach to understanding restriction theorems which is rooted in Galois theory and to showcase novel insights to period integrals which follow from it.

Periods of Calabi-Yau (CY) families are a central object of study in mirror symmetry, as the local Torelli theorem \cite{local-torelli-complete-intersections,local-torelli-cyclic-covers} essentially implies that they dictate the ``B-model" information in the mirror correspondence. As a result, periods are known to encode enumerative \cite{CANDELAS199121, cox-katz} and arithmetic \cite{Lian1996, doran1998picardfuchsuniformizationmodularitymirror, Yui2013} information about CY families. While periods of many CY families are known to be governed by GKZ generalized hypergeometric systems \cite{batyrev-vmhs,Hosono1996,hosono2019k3,HLY2019}, the best-understood examples occur in non-canonically chosen gauge-fixed subfamilies of larger universal families. An illustrative example is to contrast the Weierstrass family of elliptic curves $y^2 = x^3 + ax^2 + bx + c$, for which the period $\int_\gamma \frac{dx}{y}$ is classically and explicitly known to be closely related to the Gauss hypergeometric function $\leftindex_2{F}_1$, with the cubic family of elliptic curves given as hypersurfaces in $\P^2$, for which there is also an explicit formula thanks to the theory of GKZ hypergeometric equations \cite{Gelfand1989} and their general solutions \cite{Hosono1996}.

In order to explicitly understand the periods of larger, universal families of CY varieties such as the full cubic family in $\P^2$, one needs a procedure to lift the periods from a low-dimensional subvariety to the entire moduli space. A first step in this direction was taken by Lian and Yau \cite{Lian_2012}, who derived a Poincar\'e residue method to canonically normalize the volume forms on Calabi-Yau complete intersections in a fixed compact manifold $X$. From this canonical normalization, they constructed certain $\D$-modules, called tautological systems, specifically tailored to govern the periods. Whether or not the period sheaf exactly coincides with the solution sheaf of the tautological system depends on the cohomology of $X$ and a choice of extra equivariance data baked into the tautological system, but is known to be the case in all examples we will consider in this article \cite{BHLSY-holonomic-rank-problem,HLZ-periods-riemann-hilbert-correspondence,LLY-calabi-yau-fractional-complete-intersections}. Our key insight on periods in this article is that this extra equivariance data can be leveraged using invariant theory to lift hypergeometric period formulas from low-dimensional subfamilies to the entire moduli space.

\subsection{Structure of the article}

The structure of this article is as follows. In Section \ref{sec:two-restriction-properties}, we will define the Chevalley restriction property, which captures the essence of previous restriction theorems, and the Galois restriction property, which generalizes the former to the level of function fields. These properties will be the main focus of this article. 

In Section \ref{sec:parametrizing-extension-subvarieties}, we focus on the case of representations whose $G$-invariant rational functions are fractions of $G$-invariant polynomials. Such representations are sometimes called visible (see \cite{Renner2012}), though this adjective also has other common meanings. In this setting, we construct a scheme $\M$ over a large algebraically closed field $K$ whose closed points parametrize so-called extension subvarieties, which are candidates for the Galois and Chevalley restriction properties. A key point is that the construction canonically embeds the coordinate rings and function fields of all extension subvarieties into $K$, which allows us to phrase the relationship between the Galois and Chevalley restriction properties in terms of an intersection condition on subrings of $K$.

Then, in Section \ref{sec:CRP-algebraic}, we continue under the assumptions of the previous section and push the question of the relationship between the Chevalley and Galois restriction properties. Motivated by the intersection condition in the previous section, we define the positive closure of an integral domain $R$, a subset of $\bar{\Frac R}$ which is closely intertwined with the Chevalley restriction property. Under suitable conditions on $R$, we interpret the positive closure geometrically as a weak kind of surjectivity condition, which leads to the following equivalent characterizations of the Chevalley restriction property.
\begin{theorem*}[= Theorem \ref{thm:CRP-equivalent-conditions}]
    Under suitable conditions on $G$ and $V$, let $Y$ be a subvariety with the Galois restriction property. Then the following are equivalent.
    \begin{enumerate}
        \item $(Y, W)$ has the Chevalley restriction property.
        \item $\C[Y] \subset \PC(\C[V]^G)$.
        \item The canonical morphism $Y \to V \sslash G$ is surjective up to codimension $2$.
        \item The canonical morphism $Y \to V \sslash G$ is surjective.
    \end{enumerate}
\end{theorem*}
As an easy corollary of the Theorem, we recover a statement similar to a useful proposition of Gatti and Viniberghi (aliases of Kac and Vinberg), which highlights from an orbit-theoretic point of view why the Chevalley restriction property is a true generalization of previous restriction theorems.

Section \ref{sec:families-over-C} deals with the question of how to understand deformations of extension subvarieties. Using a technique reminiscent of Weil restriction, we construct spaces (more precisely, constructible sets) whose points are in bijection with points of $\M$ with coordinates belonging to a given $\C$-sub-vector space of $K$, and organize the corresponding extension subvarieties as a family over this base. We primarily focus on examples and open questions; this appears to be a ripe setting to investigate the properties of ``generic" extension subvarieties.

Finally, in Section \ref{sec:period-integrals}, we show the relevance of restriction properties to the study of period integrals and B-model moduli for CY families. We begin by proving two short lemmas which interpolate between the algebraic and analytic settings when given a subvariety with the Galois restriction property. In some sense, these lemmas have the flavor of a local slice theorem in the analytic category. We then prove striking explicit formulae for periods of CY double covers of $\P^n$ and elliptic curves in $\P^2$ in terms of invariant functions on their respective parameter spaces. 
\begin{theorem*}[= Theorem \ref{thm:periods-double-covers}]
    The period of the family of CY double covers of $\P^{n-1}$ is given in terms of $\SL_n(\C)$-invariant functions by the power series
    \[\Pi(Z) = \left( \frac{\left( \prod_{i=n+2}^{2n} \Delta_{2 \cdots n i} \right) \left( \prod_{\{j_1, \dots, j_{n-2}\} \subseteq \{2, \dots, n\}} \Delta_{1 j_1 \cdots j_{n-2} (n+1)} \right) }{\left( \Delta_{1 2 \cdots n} \Delta_{2 \cdots n (n+1)} \right)^{n-2}} \right)^{-\frac{1}{2}} \sum_\ell A(\ell) f_{ij}(Z)^\ell,\]
    where $\Delta_{i_1 \cdots i_n}$ are minors of an $n \times 2n$ parameter matrix, the $f_{ij}(Z)$ are explicit rational functions of the $\Delta_{i_1 \cdots i_n}$, $\ell = (\ell_{ij})$ ranges over $\Z_{\geq 0}$-valued multiindices with $2 \leq i \leq n$ and $n+2 \leq j \leq 2n$,
    \[A(\ell) = \frac{\prod_{j=n+2}^{2n} (\frac{1}{2}, \sum_{i=2}^n \ell_{ij}) \prod_{i=2}^n (\frac{1}{2}, \sum_{j=n+2}^{2n} \ell_{ij})}{(\frac{n}{2}, \sum_{i=2}^n \sum_{j=n+2}^{2n} \ell_{ij}) \prod_{i=2}^n \prod_{j=n+2}^{2n} (1, \ell_{ij})},\]
    and $(a, n) = \frac{\Gamma(a+n)}{\Gamma(a)}$. This function extends to a multivalued meromorphic function by analytic continuation.
\end{theorem*}
\begin{theorem*}[= Theorem \ref{thm:periods-hypersurfaces}]
    The period of the family of elliptic curves in $\P^2$ is given in terms of $\SL_3(\C)$-invariant functions by the power series
    \begin{align*}
        \Pi(Z) &= \sqrt[12]{\frac{729 (T^4 + 10 S^3 T^2 - 2 S^6) \pm 729  T \sqrt{(T^2 - 4 S^3)^3}}{2}} \\
        &\quad \times \sum_{n=0}^\infty \frac{\Gamma\left( n + \frac{1}{2} \right)^2}{\Gamma\left( \frac{1}{2} \right)^2 \Gamma(n + 1)^2} \left( \frac{1 \pm \sqrt{1 + 4 \sqrt[3]{\frac{729J (54 J^2 + 18 J + 1) \pm 729 J \sqrt{4J + 1}}{2}}}}{2} \right)^n
    \end{align*}
    where $S$ and $T$ are the Aronhold invariants and $J$ is the $j$-invariant of the elliptic curve (a rational function of $S$ and $T$). This power series extends to a multivalued meromorphic function by analytic continuation.
\end{theorem*} 

The history and the present article can be summarized in the following flowchart. Here, an arrow indicates inspiration or generalization. A solid box or arrow indicates a previous result and a dashed box or arrow indicates an insight introduced in this article.
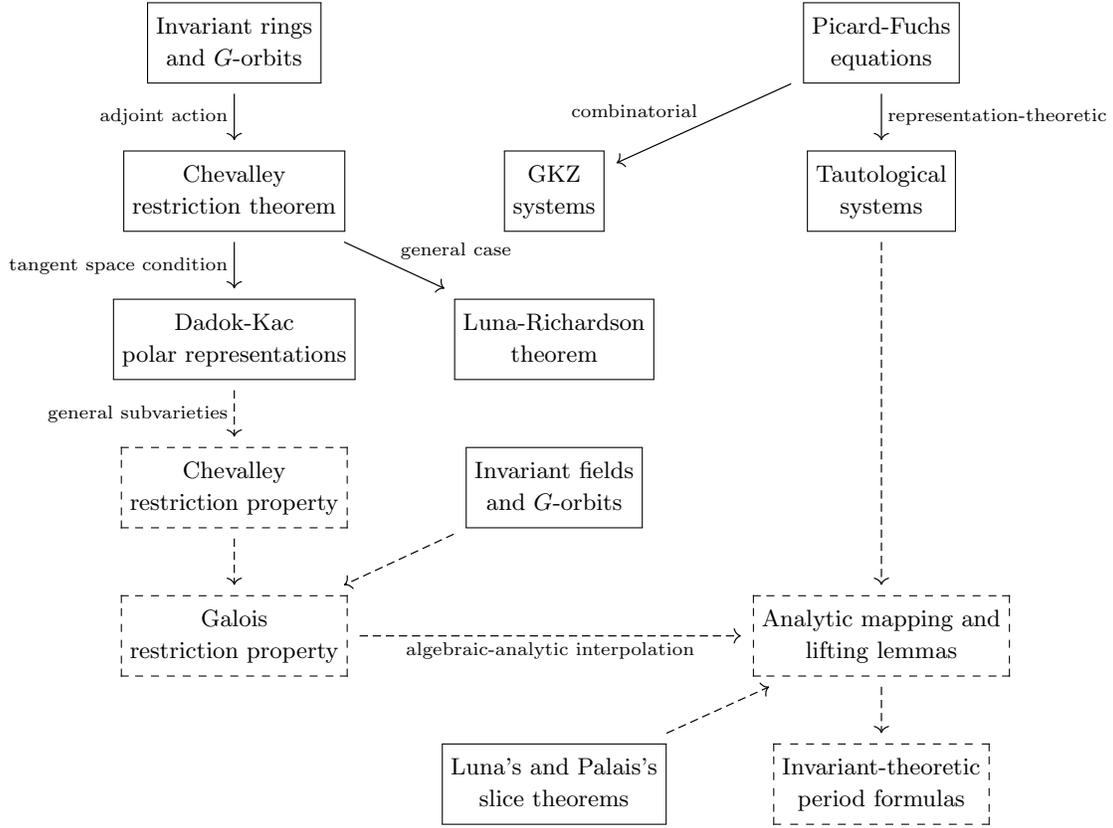
\begin{figure}[H]
    % https://tikzcd.yichuanshen.de/#N4Igdg9gJgpgziAXAbVABwnAlgFyxMJZABgBpiBdUkANwEMAbAVxiRAB125JZgBJMPQBOWOmBwACEWADmcAL7AxUCQBIA4qoC0EIQCNcCkPNLpMufIRQBGUgGYqtRizadu0GP0F0RYyQDMsGAYoBSUwFQ1tXQMcIxMzbDwCIjJrR3pmVkQOLh5PAGEACxhMhhgAT0UheBwRAGNksAkcEt0YAFt5Y1MQDCTLVNIAJgznbNz3XgAROigIAGstAGk6esUMBh8pGDQauBhxOib43v6LFJQyB2pMlxzOKHzgYtLGcqrgfbqsRssJPYQNAwIQ4Ko9RIXKwkUgAFjGWVc7CeHmA6kYECwYW+DSaAKEQJBYO6CT65iaRGG5AR90mzwACr8fFAtAAxJj1IphGAARyYx0sp0hFJQVPSt3GSKmngAKnQmDgIAwIDImQxFHAKnAcJ0hWSBpdkFT4RLEQ9kc8AIJgRgVPD1CQdOhoNBYWQSZSKBhYfx4d3lDpOvXnEVG0gAVhpE0ezwEwlE4i0rRg7XtGxB+BU-l0HSYW2D5MGNhGUalzwAMkwbVoAEq-IrMuAERTJ9pdCH6qFEWzipxmumo9TLABaGq1Oo6BYN0NskdNtLcFardAA5HAPREJPTGHQsWuNd76jAWm0apOSY4YFAZPAiKB-ASOkgyCBFUgqX2F+wdQAPHBKKAACtMXED0-mbEBqC2PRgnpQtLhAEQZCKHAOwfCAn0QD830QOx52jb8YD-YAcDEG9QLgNA1mPeoCCgKFuigugYIYODpzYJCULQx8kDw18ICQWF8KRX9-3IkFGAkOAmD0eMYDweBGJAb0wAmKA6DgEooEg5TmNg+CrEQrBkNQ0l0MwgA2agcIAdmE81RK+XZ9kOUimiTU95N+JToP09icnKX1uIwpA7P4pAAA4oLdNSNK04LMPDayBMQKLPwIxzGBvPQhF3eotDEW17QkN0dSETYBQg6LVLYdTNKvHTfNYgyOOMrizJ41LkqQABOarYvq7SOpCxArPCxBrFsdKRKI-9aI6AwbUVXx1UavTmv8oyTIS99uoml87gy2bgHE3KGAkeoNJgElenMwS9qS5SYtquKGuGzCppwx6VIG+L3qQaxsJStKfpewbjAoeQgA
\begin{tikzcd}
\snode{Invariant rings}{and $G$-orbits} \arrow[d, "\text{adjoint action}"']                                           &                                                                & \snode{Picard-Fuchs}{equations} \arrow[d, "\text{representation-theoretic}"] \arrow[ld, "\text{combinatorial}"'] \\
\snode{Chevalley}{restriction theorem} \arrow[d, "\text{tangent space condition}"'] \arrow[rd, "\text{general case}"] & \snode{GKZ}{systems}                                           & \snode{Tautological}{systems} \arrow[ddd, dashed]                                                                \\
\snode{Dadok-Kac}{polar representations} \arrow[d, "\text{general subvarieties}"', dashed]                            & \snode{Luna-Richardson}{theorem}                               &                                                                                                                  \\
\dnode{Chevalley}{restriction property} \arrow[d, dashed]                                                             & \snode{Invariant fields}{and $G$-orbits} \arrow[ld, dashed]    &                                                                                                                  \\
\dnode{Galois}{restriction property} \arrow[rr, "\text{algebraic-analytic interpolation}"', dashed]                   &                                                                & \dnode{Analytic mapping and}{lifting lemmas} \arrow[d, dashed]                                                   \\
                                                                                                                      & \snode{Luna's and Palais's}{slice theorems} \arrow[ru, dashed] & \dnode{Invariant-theoretic}{period formulas}                                                                    
\end{tikzcd}
\end{figure}
Although we hope that the reader will be convinced that invariant theory and period integrals have much to offer each other, we see this article as just the beginning of the story. Throughout the text, we provide numerous examples and point out questions for future study. We hope that this will alert others to the many interesting questions waiting to be answered in this new area.

\subsection{Notations and conventions.} 
In this article, all reductive groups are algebraic groups over $\C$ and all their representations are $\C$-vector spaces. If $X$ is a variety over a field $k$, we use $k[X]$ to denote the coordinate ring of $X$ and $\hat{X} = \Spec k[X]$ to mean the scheme-theoretic analogue of $X$. We use $\C[V]^G$ to denote the subring of $\C[V]$ comprising $G$-invariant functions, and $\C(V)$ and $\C(V)^G $ to mean the function field of $V$ and its subfield of $G$-invariant functions. For $f \in k[X]$, we use $\V(f)$ to mean the vanishing locus of $f$ in the variety $X$. In Sections \ref{sec:parametrizing-extension-subvarieties} and \ref{sec:CRP-algebraic} only, we take $V$ to be a representation satisfying $\C(V)^G = \Frac (\C[V]^G)$. $\Frac R$ denotes the field of fractions of an integral domain $R$. We use the adjective ``factorial" to mean that $R$ is a UFD. Finally, if $R$ is a reduced finitely-generated $k$-algebra, we use $\mSpec R$ to mean the variety of closed points of $\Spec R$. 

\section{Two restriction properties} \label{sec:two-restriction-properties}

In this section, we will push the results of Chevalley, Dadok-Kac, and Luna-Richardson further by defining two restriction properties that generalize the historical results outlined in Section \ref{sec:introduction}. Given a representation $V$ of a reductive group $G$, we are interested in restriction properties where $G$-invariants on $V$ restrict to invariant functions on a subvariety $Y$ under the action of a finite group $W$. Our first definition, on the level of invariant rings, aligns closely with the previous results in Section \ref{sec:introduction}.

\begin{definition} \label{def:chevalley-restriction-property}
    Let $Y \subset V$ be a closed subvariety and $W$ a finite group acting faithfully on $Y$ by regular automorphisms. We say that $(Y, W)$ has the Chevalley restriction property if $\res_{V \to Y} : \C[V]^G \to \C[Y]^W$ is an isomorphism.
\end{definition}

In this language, the Chevalley restriction theorem says that any Cartan subalgebra of $\g$ has the Chevalley restriction property, and the Dadok-Kac theorem says that any Cartan subspace of a polar representation has the Chevalley restriction property. However, there are plenty of cases not covered by these two theorems. In the example of $\SL_2(\C)$ acting on $M_{2 \times 2}(\C)$, the subvariety $Y_1 = \{z_{12} = z_{21} = 0, z_{22} = 1\}$ is not a subspace, but has the Chevalley restriction property with $W_1$ the trivial group. The subvariety $Y_2 = \{z_{12} = z_{21} = 0, z_{11} = z_{22}^2\}$ is not even a degree-one subvariety, but has the Chevalley restriction property under the action of $W_2 = \Z / 3\Z$ by $(z_{11}, z_{22}) \mapsto (\zeta^2 z_{11}, \zeta z_{22})$, where $\zeta$ is a third root of unity.

As we will see later, restriction theorems on function fields play an important role in the study of period integrals, motivating our second definition.

\begin{definition} \label{def:galois-restriction-property}
    Let $Y \subset V$ be a closed subvariety and $W$ a finite group acting faithfully on $Y$ by birational automorphisms. We say that $(Y, W)$ has the Galois restriction property if $\res_{V \to Y} : \C(V)^G \to \C(Y)^W$ is an isomorphism.
\end{definition}

Note that in the setting of this definition, the extension $\C(Y) / \C(Y)^W$ is automatically Galois, which justifies the terminology. A question which will occupy much of our attention is the precise relationship between these two restriction properties; the following is a first result in this direction.

\begin{proposition} \label{prop:chevalley-implies-galois}
    Suppose that $V$ is a representation satisfying $\C(V)^G = \Frac(\C[V]^G)$. If $(Y, W)$ has the Chevalley restriction property, then it has the Galois restriction property.
\end{proposition}
\begin{proof}
    Because $\res_{V \to Y} : \C[V]^G \to \C[Y]^W$ is an isomorphism, we have an injective map $\C[V]^G \to \C(Y)^W$ by embedding $\C[Y]^W$ into $\C(Y)^W$. This induces a homomorphism of fields $\res_{V \to Y} : \C(V)^G \to \C(Y)^W$ by the universal property of fraction fields. Since $W$ is a finite group, $\C(Y)^W = \Frac \C[Y]^W$, so surjectivity of the map of invariant rings implies that $\res_{V \to Y} : \C(V)^G \to \C(Y)^W$ is surjective, thus an isomorphism.
\end{proof}

Again in the example of $\SL_2(\C)$ acting on $M_{2 \times 2}(\C)$, now consider the subvariety $Y_3 = \{z_{12} = z_{21} = 0, z_{11} z_{22}^2 = 1\}$. Here, $\C[Y_3] \cong \frac{\C[z_{11}, z_{22}]}{(z_{11} z_{22}^2 - 1)} \cong \C[z_{22}, z_{22}^{-1}]$ and $\C(Y_3) = \Frac \C[Y] \cong \C(z_{22})$. Taking $W_3$ to be the trivial group, the restriction of the determinant $\Delta$ which generates $\C(V)^G$ to $\C(Y_3)^{W_3} = \C(Y_3)$ is the generator $z_{22}^{-1}$. However, on the level of invariant rings, $\res_{V \to Y} : \C[V]^G \to \C[Y]^{W_3}$ is not an isomorphism: it is not surjective because its image does not contain any positive powers of $z_{22}$. Therefore this is a simple example of a subvariety which has the Galois restriction property, but not the Chevalley restriction property.

\section{Parametrizing extension subvarieties} \label{sec:parametrizing-extension-subvarieties}

In this section, we will take $V$ to be a representation of a reductive group $G$ with the property that $\C(V)^G = \Frac(\C[V]^G)$. This condition is automatic in the case where $G$ is connected and semisimple. We set $\hat{V} = \Spec \C[V]$ and $K = \bar{\C(V)^G}$. We denote $\hat{V}_K = \hat{V} \times_{\Spec \C} \Spec K$ the base change of $\hat{V}$ to $K$, and $\pi : \hat{V}_K \to \hat{V}$ the induced map.

Based on Proposition \ref{prop:chevalley-implies-galois}, it seems plausible that it should be possible to upgrade the Galois restriction property to the Chevalley restriction property under certain conditions. We will see that in order to understand the interplay between these two restriction properties, we need to canonically embed $\C[V]^G$ and $\C[Y]$ in a common ambient ring. This motivates the following definition.

\begin{definition}
    The extension parameter scheme of $V$ is the closed subscheme
    \[\M = \hat{V} \times_{\hat{V \sslash G}} \Spec K = \Spec (\C[V] \otimes_{\C[V]^G} K).\]
    of $\hat{V}_K$. We use $\iota : \M \to \hat{V}_K$ to denote the closed embedding.
\end{definition}

$\M$ is nonempty because $\C[V]$ and $K$ are nontrivial torsion-free $\C[V]^G$-modules. In coordinates, given a generating set $\{P_1, \dots, P_m\}$ of $\C[V]^G$, $\M$ has the explicit description
\[\M = \Spec \left( \frac{\C[V] \otimes_\C K}{(P_1 \otimes 1 - 1 \otimes P_1, \dots, P_m \otimes 1 - 1 \otimes P_m)} \right).\]

\begin{example}
    Let $G = \SL_2(\C)$ act on $V = M_{2 \times 2}(\C)$ in the usual way. Then $\C[V]^G = \C[\Delta]$ where $\Delta$ denotes the determinant, and taking $K = \bar{\C(\Delta)}$, we get $\M = \Spec \frac{K[z_{11}, z_{21}, z_{12}, z_{22}]}{z_{11} z_{22} - z_{12} z_{21} - \Delta}$, a three-dimensional closed subscheme of $\A_K^4$.
\end{example}

\begin{remark}
    In the case where $\C(V)^G \neq \Frac(\C[V]^G)$, a similar construction of $\M$ can be made with the trade-off that $\M$ will only be a locally closed rather than closed subscheme. For example, consider the scaling action of $G = \C^\times$ on $V = \C^2$. This representation has $\C[V]^G = \C$ but $\C(V)^G = \C(\frac{x}{y})$, so we set $K = \bar{\C(T)}$ where $T$ denotes the generator $\frac{x}{y}$ of $\C(V)^G$. In this case, the ``right" extension parameter scheme is $\M = \Spec \frac{K[x, y, x^{-1}, y^{-1}]}{(x - T y)}$, i.e. the set of closed points is $\{x = Ty\} \setminus \{(0, 0)\} \subset K^2$.
\end{remark}

For the rest of this section, let $p$ be a closed point of $\M$.
\begin{definition}
    The extension subscheme $\hat{Y}_p$ associated to $p$ is the closed irreducible subscheme of $\hat{V}$ defined by the scheme-theoretic image of $\pi \circ \iota : \{p\} \to \hat{V}$. In other words, $\hat{Y}_p = \bar{\{\pi(\iota(p))\}} = \Spec \left( \frac{\C[V]}{\m \cap \C[V]} \right)$ where $\m \subset \C[V] \otimes_\C K$ is the maximal ideal representing $p$. Because $\m \cap \C[V]$ is a prime ideal, $\C[Y_p] = \frac{\C[V]}{\m \cap \C[V]}$ is the coordinate ring of the variety of closed points of $\hat{Y}_p$, which we call the extension subvariety associated to $p$ and denote by $Y_p$.
\end{definition}

The extension parameter scheme has two main uses: to parametrize extension subvarieties of $V$ in geometric terms, and to canonically embed their coordinate rings into $K$. In Section \ref{sec:families-over-C}, the canonical embedding of coordinate rings into $K$ will be used again to understand deformations of extension subvarieties.

\begin{proposition} \label{prop:coord-ring-embed-in-K}
    $\C[Y_p]$ is canonically isomorphic to the image of the composition
    \[\C[V] \to \C[V] \otimes_\C K \to \frac{\C[V] \otimes_\C K}{\m} \cong K\]
    where the right-hand isomorphism is the unique such isomorphism over $K$. The image of this map is a subring of $K$ containing $\C[V]^G$. The function field $\C(Y_p)$ of $Y_p$ is canonically isomorphic to a subfield of $K$ which is a finite extension of $\C(V)^G$.
\end{proposition}
\begin{proof}
    The kernel of the composition $\C[V] \to \C[V] \otimes_\C K \to \frac{\C[V] \otimes_\C K}{\m}$ is exactly $\m \cap \C[V]$, proving the first statement. For the second, by definition of $\M$, it is clear that $P \otimes 1 - 1 \otimes P \in \m$ for each $P \in \C[V]^G$, and therefore the image of $P$ in $K$ under this composition is $P$. The third statement follows because the canonical inclusion of $\C[Y_p]$ into $K$ induces a canonical inclusion of $\C(Y_p)$ into $K$ by the universal property of fraction fields. Because $K = \bar{\C(V)^G}$, it is immediate that $\C(Y_p)$ is an algebraic extension of $\C(V)^G$. In particular, $\C(Y_p)$ is a finitely-generated algebraic extension of $\C(V)^G$ (for it is generated over $\C(V)^G$ by the image of a generating set of $\C[V]$ under the homomorphism $\C[V] \to \C[Y_p]$), and therefore it is a finite extension.
\end{proof}

The proof of Proposition \ref{prop:coord-ring-embed-in-K} shows that the restriction maps $\res_{V \to Y_p}$ sending $\C[V]^G \to \C[Y_p]$ and $\C(V)^G$ to $\C(Y_p)$ are well-defined and injective: they are precisely the canonical inclusions of $\C[V]^G$ into $\C[Y_p]$ and $\C(V)^G$ into $\C(Y_p)$ described in the proposition. This allows us to show that the extension subvarieties are good candidates for restriction properties.

\begin{proposition} \label{prop:M-and-subvarieties-correspondence}
	The extension subvariety associated to each closed point in $\M$ is a closed subvariety $Y \subset V$ so that $\res_{V \to Y} : \C[V]^G \to \C[Y]$ is injective, and $\C[Y]$ is an algebraic ring extension of $\C[V]^G$. Conversely, every such subvariety is the extension subvariety induced by a (possibly non-unique) closed point of $\M$.
\end{proposition}
\begin{proof}
    The first statement immediately follows from Proposition \ref{prop:coord-ring-embed-in-K} and the following discussion. For the second, let $Y$ be such a subvariety; we will produce a closed point $p$ of $\M$ so that $Y = Y_p$. By hypothesis, the coordinate ring $\C[Y]$ embeds (possibly non-uniquely) into $K$ so that $\res_{V \to Y} : \C[V]^G \to \C[Y] \subset K$ is the identity on $\C[V]^G$. The homomorphism $\C[V] \otimes_{\C[V]^G} K \to K$ sending $f \otimes c \mapsto (\res_{V \to Y} f) c$ is surjective, and its kernel is a maximal ideal $\m \subset \C[V] \otimes_{\C[V]^G} K$ such that the composition
	\[\C[V]  \to \C[V] \otimes_{\C[V]^G} K \to \frac{\C[V] \otimes_{\C[V]^G} K}{\m} \cong K\]
	sends $\C[V]$ to the chosen embedding of $\C[Y]$ into $K$. This defines a closed point $p$ of $\hat{V}_K$ for which it is clear that $Y = Y_p$, and which furthermore belongs to $\M$ because $\res_{V \to Y} : \C[V]^G \to \C[Y]$ is injective.    
\end{proof}

\begin{corollary}
    If $(Y, W)$ has the Galois restriction property, then $Y$ is the extension subvariety associated to some closed point in $\M$.
\end{corollary}
\begin{proof}
    If $(Y, W)$ has the Galois restriction property, then $\res_{V \to Y} : \C(V)^G \to \C(Y)^W$ is an isomorphism and $\C(Y) / \C(V)^G$ is a Galois extension. From this it is easy to see that $\res_{V \to Y} : \C[V]^G \to \C[Y]$ is injective and $\C[Y]$ is an algebraic ring extension of $\C[V]^G$. The claim then follows from Proposition \ref{prop:M-and-subvarieties-correspondence}.
\end{proof}

Because our construction canonically embeds the coordinate rings and function fields of extension subvarieties into $K$, we can compare $\C(V)^G$ and $\C[Y]$ as subrings of $K$ to relate the Chevalley and Galois restriction properties.

\begin{proposition} \label{prop:crp-intersection-condition}
    Let $p \in \M$ be a closed point and suppose that the extension subvariety $Y = Y_p$ has the Galois restriction property. Then $Y$ has the Chevalley restriction property iff $\C[Y] \cap \C(V)^G = \C[Vz]^G$, where the intersection is meant as an intersection of subrings of $K$.
\end{proposition}
\begin{proof}
    Since $\res_{V \to Y} : \C(V)^G \to \C(Y)^W$ is an isomorphism, and $\C[V]^G = \C[V] \cap \C(V)^G$ and $\C[Y]^W = \C[Y] \cap \C(Y)^W$, $\res_{V \to Y}$ restricts to an injective map $\C[V]^G \to \C[Y]^W$. The Chevalley restriction property for $Y$ is then equivalent to surjectivity of this map. Thus, the Chevalley restriction property holds iff
    \[\C[V]^G = \C[Y]^W = \C[Y] \cap \C(Y)^W = \C[Y] \cap \C(V)^G\]
    inside of $K$.
\end{proof}

To illustrate this proposition, consider the final example from Section \ref{sec:two-restriction-properties}: in the representation $M_{2 \times 2}(\C)$ of $\SL_2(\C)$, consider the subvariety $Y_3 = \{z_{12} = z_{21} = 0, z_{11} z_{22}^2 = 1\}$. Inside of $K = \bar{\C(\Delta)}$, $\C[Y_3] \cap \C(V)^G = \C[\Delta, \Delta^{-1}] \cap \C(\Delta)$. This intersection equals $\C[\Delta, \Delta^{-1}]$, which is strictly larger than $\C[V]^G$.

\begin{remark}
    At this juncture, the theory already has enough strength to clarify a key aspect of previous restriction theorems by providing an intrinsic definition of the Weyl group. For example, given a Cartan subspace $\mathfrak{c}$ of a polar representation $V$, Dadok and Kac defined the Weyl group $W$ in terms of the quotient of the normalizer and stabilizer of $\mathfrak{c}$ under the $G$-action, but remarked on the difficulty of explicitly computing $W$. In our language, $(\mathfrak{c}, W)$ has the Chevalley restriction property, and consequently has the Galois restriction property by Proposition \ref{prop:chevalley-implies-galois}. Therefore $W$ is precisely the Galois group of the extension $\C(\mathfrak{c}) / \C(V)^G$, and curiously can be defined without any reference to the original group $G$. The same idea applies to the Weyl group of the Chevalley restriction theorem and the group $W$ of the Luna-Richardson construction when it is finite.
\end{remark}

We conclude this section by enunciating the precise relationship between closed points in $\M$ and extension subvarieties in $V$.

\begin{theorem} \footnote{An incorrect version of this theorem appeared in an earlier version of this article.} \label{thm:galois-orbits-subvarieties-bijection}
	Let $Y$ be an extension subvariety of $V$. Then there is a natural bijection between the set
	\[A = \{\text{closed points }p \in \M \st Y_p = Y\}\]
	and the set
	\[B = \{\text{field embeddings } \sigma : \C(Y) \to K \st \sigma|_{\C(V)^G} = \Id_{\C(V)^G}\}.\]
\end{theorem}
\begin{proof}
    Each closed point in $A$, given by a maximal ideal $\m \subset \C[V] \otimes_{\C[V]^G} K$, induces a map
	\[\C[V] \to \C[V] \otimes_{\C[V]^G} K \to \frac{\C[V] \otimes_{\C[V]^G} K}{\m} \cong K.\]
    The kernel of this map is $\m \cap \C[V]$, so we get an injection
	\[\C[Y] = \C[Y_p] = \frac{\C[V]}{\m \cap \C[V]} \to K\]
	which is the identity on $\C[V]^G$. By the universal property of fraction fields, this gives a field embedding $\sigma : \C(Y) \to K$ which is the identity on $\C(V)^G$. Conversely, any element $\sigma \in B$ induces a map
	\begin{align*}
		\C[V] \otimes_{\C[V]^G} K &\to K \\
		f \otimes c &\mapsto \sigma(\res_{V \to Y} f) c
	\end{align*}
	which is well-defined precisely because $\sigma$ is the identity on $\C(V)^G$. Since this map is surjective onto $K$, its kernel is a maximal ideal $\m \subset \C[V] \otimes_{\C[V]^G} K$, that is, a closed point of $\M$. To see that this closed point belongs to $A$, we observe that $\m \cap \C[V]$ is the kernel of the map $\C[V] \to \C[V] \otimes_{\C[V]^G} K \to K$, $f \mapsto f \otimes 1 \mapsto \sigma(\res_{V \to Y} f)$. Since $\sigma$ is injective, this kernel is precisely $\ker(\res_{V \to Y})$, the ideal of functions vanishing on $Y$.

    Finally, we check that these correspondences are inverse to one another. Given $\sigma \in B$, the corresponding element of $A$ is the maximal ideal $\m$ given by the kernel of the map $\C[V] \otimes_{\C[V]^G} K \to K$, $f \otimes c \mapsto \sigma(\res_{V \to Y} f) c$, which has the property that $\m \cap \C[V]$ is the ideal cutting out $Y$. This closed point in turn induces the map $\C[Y] = \frac{\C[V]}{\m \cap \C[V]} \to \frac{\C[V] \otimes_{\C[V]^G} K}{\m} \cong K$, $f + \m \cap \C[V] \mapsto (f \otimes 1) + \m \mapsto \sigma(f)$, so the induced embedding $\C(Y) \to K$ agrees with $\sigma$. On the other hand, given a point in $A$ represented by a maximal ideal $\m \subset \C[V] \otimes_{\C[V]^G} K$, the induced field embedding $\sigma$ is defined on $\C[Y]$ by the chain of maps $\C[Y] = \frac{\C[V]}{\m \cap \C[V]} \to \frac{\C[V] \otimes_{\C[V]^G} K}{\m} \cong K$, $f + \m \cap \C[V] \mapsto (f \otimes 1) + \m \mapsto \sigma(f)$. The element of $A$ induced from this $\sigma$ is exactly the maximal ideal given by the kernel of the map $\C[V] \otimes_{\C[V]^G} K \to \frac{\C[V] \otimes_{\C[V]^G} K}{\m}$, $r \mapsto r + \m$, which is simply $\m$.
\end{proof}

\section{The Chevalley restriction property in algebraic terms} \label{sec:CRP-algebraic}

Thanks to Proposition \ref{prop:crp-intersection-condition}, the question of when a subvariety with the Galois extension property also has the Chevalley restriction property is reduced to a question of computing intersections of subrings of $K$. Motivated by this perspective, we define the notion of the positive closure.

\begin{definition}
	Let $R$ be an integral domain. The positive closure of $R$ is
	\[\PC(R) = \{s \in \bar{\Frac R} \st R[s] \cap \Frac R = R\}.\]
\end{definition}

If $R$ is a field, then $\PC(R)$ is its algebraic closure, but explicitly computing $\PC(R)$ is generally quite difficult. The relationship of this definition to the Chevalley restriction property is the following.

\begin{proposition} \label{prop:PC-intersection-condition}
	Let $S$ be a ring such that $R \subseteq S \subseteq \bar{\Frac R}$. Then $S \cap \Frac R = R$ iff $S \subseteq \PC(R)$.
\end{proposition}
\begin{proof}
	Suppose first that $S \cap \Frac R = R$. For each $s \in S$, since $R[s] \subseteq S$, then
	\[R \subseteq R[s] \cap \Frac R \subseteq S \cap \Frac R = R\]
	and therefore $s \in \PC(R)$. Conversely, if $S \subseteq \PC(R)$, then $R[s] \cap \Frac R = R$ for every $s \in S$. Because $S = \bigcup_{s \in S} R[s]$ as subrings of $\bar{\Frac R}$, it follows that
	\[S \cap \Frac R = \left( \bigcup_{s \in S} R[s] \right) \cap \Frac R = \bigcup_{s \in S} \left( R[s] \cap \Frac R \right) = \bigcup_{s \in S} R = R.\]
\end{proof}

When $R = \C[V]^G$ (again under the assumption that $\C(V)^G = \Frac (\C[V]^G)$) and $S = \C[Y]$ where $Y$ is a subvariety with the Galois restriction property, then Propositions \ref{prop:crp-intersection-condition} and \ref{prop:PC-intersection-condition} together imply that $Y$ has the Chevalley restriction property iff $\C[Y] \subseteq \PC(\C[V]^G)$.

The rest of this section will focus on understanding some properties and characterizations of the positive closure. Our first result is a criterion for testing whether a given element of $\bar{\Frac R}$ belongs to $\PC(R)$ and resembles the definition of an integral element of a ring extension. 

\begin{lemma} \label{lem:PC-polynomial-criterion}
	Let $s \in \bar{\Frac R}$. Then $s \in \PC(R)$ iff for every polynomial $f(x) = r_k x^k + \cdots + r_1 x + r_0 \in R[x]$ satisfied by $s$ and every $p \in R$ such that $p$ divides $r_k, \dots, r_1$, then $p$ also divides $r_0$ (where divisibility is meant in $R$).
\end{lemma}
\begin{proof}
	Suppose $s \in \PC(R)$ satisfies a polynomial $f(x) = r_k x^k + \cdots + r_1 x + r_0 \in R[x]$, and there is some $p \in R$ dividing $r_k, \dots, r_1$. Then, dividing the equation $f(s) = 0$ through by $p$, we have
    \[\frac{r_k}{p} s^k + \cdots + \frac{r_1}{p} s = -\frac{r_0}{p}.\]
	By the divisibility assumption, $\frac{r_k}{p}, \dots, \frac{r_1}{p} \in R$, so the left-hand side belongs to $R[s]$, while the right-hand side belongs to $\Frac R$. Thus, since $s \in \PC(R)$, in fact $\frac{r_0}{p} \in R$, that is, $p$ divides $r_0$.
	
	Conversely, let $s \in \bar{\Frac R}$, and suppose that for every polynomial $f(x) = r_k x^k + \cdots + r_1 x + r_0 \in R[x]$ satisfied by $s$ and every $p \in R$ such that $p$ divides $r_k, \dots, r_1$, then $p$ also divides $r_0$. Let $q \in R[s] \cap \Frac R$, $q = r_k s^k + \cdots + r_1 s + r_0 = \frac{n}{d}$. Rearranging and multiplying through by $d$, we have
    \[(d r_k) s^k + \cdots + (d r_1) s + (d r_0 - n) = 0.\]
	By the divisibility assumption, since $d$ divides $d r_k, \dots, d r_1$, then $d$ also divides $d r_0 - n$. Since $d$ clearly divides $d r_0$, this means that $d$ divides $n$; therefore $q$ belongs to $R$.
\end{proof}

Like the integral closure of $R$ in $\bar{\Frac R}$, the positive closure is transitive in the following sense.

\begin{lemma}
	If $s \in \PC(R)$ and $t \in \PC(R[s])$, then $t \in \PC(R)$.
\end{lemma}
\begin{proof}
	First, note that since $s \in \PC(R) \subseteq \bar{\Frac R}$, a field, then $R[s]$ is an integral domain and $\PC(R[s])$ is well-defined. Since $R[s] \cap \Frac R = R$ and $R[s, t] \cap \Frac (R[s]) = R[s]$, we have
	\[R[s, t] \cap \Frac R = R[s, t] \cap \Frac (R[s]) \cap \Frac R = R[s] \cap \Frac R = R.\]
	Therefore $R \subseteq R[t] \cap \Frac R \subseteq R[s, t] \cap \Frac R = R$, so $R[t] \cap \Frac R = R$.
\end{proof}

However, in contrast to the integral closure, $\PC(R)$ may not be a ring. We know of no criterion to tell when $\PC(R)$ is a ring, except the simple observation above when $R$ is a field, and any such condition applicable to the Chevalley restriction property would likely be exceedingly strict. For example, even the polynomial ring $R = \C[x]$ fails to have a ring as its positive closure: take $s_1, s_2$ to be the roots of $f(y) = (x+1) y^2 - xy + 1$ in $\bar{\C(x)}$; then $s_1$ and $s_2$ both belong to $\PC(R)$, but their sum $\frac{x}{x+1}$ does not. Even explicitly describing the positive closure of a simple example like $R = \C[x]$ seems quite difficult. 

Because $\PC(R)$ is not a ring, in order to check whether $Y$ has the Chevalley restriction property, it is not enough to just check that the generators of $\C[V]$ map into $\PC(\C[V]^G)$ under the map $\C[V] \to \C[Y] \to K$. One main takeaway from this is that the Chevalley restriction property is rather poorly behaved from the algebraic viewpoint and consequently very difficult to detect in families of extension subvarieties. We will see evidence in Section \ref{sec:families-over-C} that the Galois restriction property is much more well-behaved in this regard.

In light of the above, we switch perspectives now to pursue a geometric formulation of the positive closure. It's easy to see that the natural morphism from any extension subvariety $Y$ to $V \sslash G$ is dominant, but the Chevalley restriction property is in fact related to the following stronger notion.

\begin{definition}
	Let $\pi : X \to Y$ be a morphism of varieties over $k$. We say $\pi$ is surjective up to codimension two if 
	\[\min(\codim_Y Z_1, \dots, \codim_Y Z_k) \geq 2\]
	where $Z_1, \dots, Z_k$ are the irreducible components of $\bar{Y \setminus \im \pi}$.
\end{definition}

Because of our intention of studying invariant rings, we will work in the case when $R$ is a reduced factorial finitely-generated $\C$-algebra, for example a polynomial ring over $\C$. (The factoriality assumption can probably be weakened.) 

\begin{remark}
    The factoriality condition is automatic for polar representations \cite{Dadok1985} and whenever $G$ is a connected semisimple group \cite{schwarz-lifting-smooth-homotopies}.
\end{remark}

\begin{lemma}
	Let $R$ be a a reduced factorial finitely-generated $\C$-algebra and let $s \in \bar{\Frac R}$. Then $s$ has a minimal polynomial over $R$, i.e. the kernel of the homomorphism $R[x] \to R[s]$ is of the form $(m(x))$ for some $m(x) \in R[x]$.
\end{lemma}
\begin{proof}
	First, since $R[s] \subset \bar{\Frac R}$, a field, then $R[s]$ is a reduced finitely-generated $\C$-algebra, and we may set $X = \mSpec (R[s])$ and $Y = \mSpec R$. Since $R$ is a UFD, so is the polynomial ring $R[x]$. Let $I$ be the kernel of the homomorphism $R[x] \to R[s]$; since the target is an integral domain, $I$ is prime. We will show that in fact $I$ is principal. Since the homomorphism $R \to R[s]$ induces a homomorphism of fields $\Frac R \to \Frac(R[s])$, the induced morphism $X \to Y$ is dominant and therefore $\dim X \geq \dim Y$. Hence,
	\[\operatorname{height}(I) = \codim_{\mSpec(R[x])} X = \dim(\mSpec(R[x])) - \dim X = \dim Y - \dim X + 1\]
	is either 0 or 1. Since $I \neq (0)$, however, we must have $\operatorname{height}(I) = 1$. In other words, $I$ is a height-one prime ideal in the UFD $R[x]$ and is therefore principal, $I = (m(x))$ for some $m(x) \in R[x]$.
\end{proof}

\begin{theorem} \label{thm:PC-surjective-codim-2}
	Let $R$ be a a reduced factorial finitely-generated $\C$-algebra and let $Y = \mSpec R$. Let $s \in \bar{\Frac R}$, $X = \mSpec(R[s])$, and $\pi : X \to Y$ the morphism induced by the inclusion $R \to R[s]$. Then $s \in \PC(R)$ iff $\pi$ is surjective up to codimension two.
\end{theorem}
\begin{proof}
	Suppose first that $\pi$ is not surjective up to codimension two. Let $Z = \V(d)$ be an irreducible component of $\bar{Y \setminus \im \pi}$ of codimension $0$ or $1$ in $Y$; in particular, $d$ is neither zero nor a unit. If $d$ were to divide all the coefficients of the minimal polynomial $m(x) = r_k x^k + \cdots + r_1 x + r_0 \in R[x]$ of $s$, then since $R[s]$ is a domain and $d$ is nonzero, $s$ would satisfy the polynomial $\frac{r_k}{d} x^k + \cdots + \frac{r_1}{d} x + \frac{r_0}{d}$. But this polynomial does not belong to $I = (m(x))$ because $d$ is not a unit. This is a contradiction, so $d$ does not divide all the coefficients of $m(x)$. The complement of the image of $\pi$ is the set of points in $Y$ where the equation $m(x) = 0$ cannot be satisfied, i.e. where $r_k = \cdots = r_1 = 0$ but $r_0 = 0$. Therefore
	\[\bar{Y \setminus \im \pi} = \bar{\V(r_k, \dots, r_1) \setminus \V(r_0)}.\]
	Since $Z \subseteq \bar{Y \setminus \im \pi}$, in particular we have $Z \subseteq \bar{\V(r_k, \dots, r_1)} = \V(r_k, \dots, r_1)$ and therefore $d$ divides $r_k, \dots, r_1$. It follows that $d$ cannot divide $r_0$, for then it would divide all the coefficients of $m$. By Lemma \ref{lem:PC-polynomial-criterion}, we then conclude that $s \notin \PC(R)$.
	
	Conversely, suppose that $s \notin PC(R)$, so there is some $q_k s^\ell + \cdots + q_1 s + q_0 = \frac{n}{d}$ which belongs to $(R[s] \cap \Frac R) \setminus R$. We may assume $n$ and $d$ share no irreducible factors. Clearing denominators and rearranging, we get the equation
	\[f(s) = d q_k s^\ell + \cdots + d q_1 s + d q_0 - n = 0.\]
	Set $Z = \V(d)$. Since $\frac{n}{d} \notin R$ and $n$ and $d$ share no irreducible factors, $\bar{\V(d) \setminus \V(n)} = \V(d)$. But on $\V(d) \setminus \V(n)$, the equation $f(x) = 0$ has no solutions, so $\V(d) \setminus \V(n) \subseteq Y \setminus \im \pi$ and therefore $\V(d) = \bar{\V(d) \setminus \V(n)} \subseteq \bar{Y \setminus \im \pi}$. Since $\V(d)$ has codimension 1 in $Y$, $\pi$ is not surjective up to codimension two.
\end{proof}

\begin{proposition} \label{prop:subring-of-positive-closure-iff-surjective-up-to-codimension-2}
	Let $R$ be a a reduced factorial finitely-generated $\C$-algebra and let $Y = \mSpec R$. Let $S$ be a ring such that $R \subseteq S \subseteq \bar{\Frac R}$ and $S$ is finitely generated over $R$. Let $X = \mSpec S$, and $\pi : X \to Y$ the morphism induced by the inclusion $R \to S$. Then $S \subseteq \PC(R)$ iff $\pi$ is surjective up to codimension two.
\end{proposition}
\begin{proof}
	If $S \subseteq \PC(R)$, then $S \cap \Frac R = R$ inside of $\Frac S$. Interpreting this in terms of regular and rational functions, this is the same as saying that $\C[X] \cap \pi^* \C(Y) = \pi^* \C[Y]$ inside of $\C(X)$, where $\pi^*$ denotes the pullback of functions from $Y$ to $X$. Suppose by way of contradiction that $Y$ is not surjective up to codimension two, so there is an irreducible component $Z = \V(d) \subseteq \bar{Y \setminus \im \pi}$ of codimension 1 in $Y$, where $d$ is neither zero nor a unit. Let $r = \frac{1}{d} \in \C(Y)$; in particular $r \notin \C[Y]$. Since $r$ is regular outside of $Z$ and $\im \pi$ does not meet $Z$, $\pi^* r \in \C(X)$ is a regular function on $X$. In other words, $\pi^* r \in \C[X] \cap \pi^* \C(Y)$, so $\pi^* r \in \pi^* \C[Y]$, meaning that there is some $r' \in \C[Y]$ so that $r$ and $r'$ agree on $\im \pi$. But $\pi$ is dominant, so this forces $r$ and $r'$ to agree everywhere. This is a contradiction because $r$ is not regular on $Y$, and therefore $Y$ is surjective up to codimension $2$.
	
	Conversely, suppose that $S \nsubseteq \PC(R)$, so there exists some $s \in S$ so that $R[s] \cap \Frac R \neq R$. Observe that $\pi$ factors through $\mSpec(R[s])$: $\mSpec S \to \mSpec (R[s]) \to \mSpec R$. Since $s \notin \PC(R)$, the latter map is not surjective up to codimension two by Theorem \ref{thm:PC-surjective-codim-2}, and therefore the composition isn't either.
\end{proof}

This geometric understanding of the positive closure now allows us to recharacterize the Chevalley restriction property.

\begin{theorem} \label{thm:CRP-equivalent-conditions}
    Suppose $\C[V]^G$ is a UFD and let $(Y, W)$ have the Galois restriction property. Then the following are equivalent.
    \begin{enumerate}
        \item $(Y, W)$ has the Chevalley restriction property.
        \item $\C[Y] \subset \PC(\C[V]^G)$.
        \item The canonical morphism $Y \to V \sslash G$ is surjective up to codimension $2$.
        \item The canonical morphism $Y \to V \sslash G$ is surjective.
    \end{enumerate}
\end{theorem}
\begin{proof}
    The equivalence of conditions (1) and (2) follows from Proposition \ref{prop:crp-intersection-condition} and the definition of the positive closure. The equivalence of (2) and (3) is exactly Proposition \ref{prop:subring-of-positive-closure-iff-surjective-up-to-codimension-2}. Condition (1) implies condition (4) because the map $Y \to Y \sslash W$ is surjective and the Chevalley restriction property implies $Y \sslash W \cong V \sslash G$, and it is obvious that condition (4) implies condition (3).
\end{proof}

\begin{corollary} \label{cor:grp-intersect-every-closed-orbit}
    Suppose $\C[V]^G$ is a UFD and let $(Y, W)$ have the Galois restriction property. If $Y$ intersects every closed $G$-orbit, then $Y$ has the Chevalley restriction property.
\end{corollary}
\begin{proof}
    Since each fiber of $V \to V \sslash G$ contains a unique closed orbit, the map $Y \to V \sslash G$ is surjective, and the claim follows by Theorem \ref{thm:CRP-equivalent-conditions}.
\end{proof}

Corollary \ref{cor:grp-intersect-every-closed-orbit} bears much resemblance to an important result of Gatti and Viniberghi; compare with Proposition 8 of \cite{GATTI1978137}. Their work is a key ingredient in the work of Dadok and Kac on polar representations, who prove their restriction theorem by recognizing that every Cartan subspace intersects all closed $G$-orbits and only closed $G$-orbits. We have traded Gatti and Viniberghi's hypothesis that $Y$ only intersects closed $G$-orbits for the Galois restriction property, and our statement deals with an arbitrary finite group $W$ arising as the Galois group of $\C(Y) / \C(V)^G$ whereas their group $W$ is required to be exhibited as a stabilizer of $Y$. While Gatti and Viniberghi clearly recognized the importance of the invariant field in restriction theorems -- see their Proposition 5 and the proof of their Proposition 8 -- it appears that the role of Galois theory has gone unnoticed.

\begin{example}
    The converse of Corollary \ref{cor:grp-intersect-every-closed-orbit} does not hold. For example, take the action of $G = \C^\times$ on $V = \C^2$ by $c \cdot (z_1, z_2) = (c z_1, c^{-1} z_2)$. Since $\C[V]^G = \C[z_1 z_2]$ and $\C(V)^G = \Frac(\C[V]^G)$, all our standing assumptions are satisfied. The subvariety $Y = \{z_1 = 1\}$ has the Chevalley restriction property, but does not intersect the unstable closed orbit (the origin in $\C^2$). Note however that $Y$ does meet a non-closed orbit in the same fiber as the origin (the punctured line $\{z_2 = 0, z_1 \neq 0\}$).
\end{example}

\begin{remark}
     From the orbit-theoretic point of view, the previous example highlights a significant contrast between our work and that of Dadok and Kac. Specifically, Cartan subspaces of a polar representation necessarily meet every closed orbit and only closed orbits, whereas subvarieties with the Chevalley restriction property are free to meet non-closed orbits, as long as they meet every fiber of $V \to V \sslash G$.
\end{remark}

\section{Complex-parameter families of extension subvarieties} \label{sec:families-over-C}

Considering the extension subvarieties of $V$ parametrized by the closed points of the $K$-scheme $\M$, it seems natural to ask questions about the properties of a ``generic" extension subvariety. However, $\M$ appears too large to yield useful information on such genericity questions because any dense open in $\M$ will parametrize extension subvarieties whose function fields cover $K$. In this section, we will present examples of finite-dimensional complex families of extension subvarieties which illustrate the types of questions we are interested in.

Viewing the variety of closed points of $\M$ as a closed subvariety of $K^n$, our objective is cut $\M$ down to a finite-dimensional complex space by considering its intersection with $A^n$, where $A$ is a finite-dimensional $\C$-sub-vector space of $K$. The obstruction is that $\M \cap A^n$ does not make sense as a scheme-theoretic object because $\M$ is a $K$-scheme. To accomplish this goal, we need a strategy for passing between $\C$-schemes and $K$-schemes. The method we present below can be regarded as an analogue of Weil restriction (see \cite{Poonen2017}, or \cite{Weil1982} for a more classical description by Weil himself) and will be the subject of future work.

On the level of rings, the closed points we are interested in are precisely those for which the image of the $\C$-linear map $V^* \to \C[V] \to \C[V] \otimes_{\C[V]^G} K \to K$ is contained in $A$. Without loss of generality we may assume that the $\C$-subalgebra $\langle A \rangle$ of $K$ generated by the elements of $A$ contains $\C[V]^G$; if not, there are no such closed points. As such linear maps are elements of $\Hom_\C(V^*, A) \cong V \otimes_\C A$, the space $\mathcal{B}$ will be a subset of the scheme-theoretic analog $\hat{V \otimes_\C A}$ of this vector space. To view $\hat{V \otimes_\C A}$ as living inside of $K^n$, we need to base change to $K$, upon which we have a natural morphism of $K$-schemes $\mu : (\hat{V \otimes_\C A})_K \to \hat{V}_K$ given on closed points by $v \otimes a \mapsto a v$.

Thus, we may consider the inverse image $\mu^{-1}(\M) = (\hat{V \otimes_\C A})_K \times_{\hat{V}_K} \M$. Suppose that $\mu^{-1}(\M)$ is cut out by an ideal $I$ of $\C[V \otimes_\C A] \otimes_\C K$ whose generators $f_i$ are expressed in terms of elements $a_i$ of $\langle A \rangle$ and generators $x_i$ of $\C[V \otimes_\C A]$. From the geometric perspective, the $a_i$ represent regular functions on $S = \Spec \langle A \rangle$ and the $f_i$ represent algebraic combinations of these algebraic functions over $\C$ once complex values of the $x_i$ have been chosen. 

Thus, passing back from this $K$-scheme to a $\C$-scheme amounts to understanding the set of possible values of the $x_i$ over $\C$ that make the elements defining $I$ identically zero. To do this, we take the scheme-theoretic image $\mathcal{A}$ of the morphism $\varphi : \mu^{-1}(\M) \to (\hat{V \otimes_\C A}) \times_\C S$ given on the level of rings by
\begin{align*}
    \C[V \otimes_\C A] \otimes_\C \langle A \rangle &\to (\C[V \otimes_\C A] \otimes_\C K) / I \\
    r \otimes s &\mapsto (r \otimes s) + I.
\end{align*}
Geometrically, $\mathcal{A}$ is the common zero locus of the $f_i$ in $(\hat{V \otimes_\C A}) \times_\C S$. It then follows that the set of $x_i$ which make the $f_i$ identically zero on $S$ is the locus in $\hat{V \otimes_\C A}$ over which the fibers of $\mathcal{A} \to \hat{V \otimes_\C A}$ are the complete $S$, i.e. for which the fibers are of full dimension $\dim S$. By \cite[Tag 05F6]{stacks-project}, this is a constructible set in $\hat{V \otimes_\C A}$ which we call $\mathcal{B}$ (though in all examples we present $\mathcal{B}$ will be an honest closed subvariety), and its closed points are in natural bijection with the closed points of $\M$ lying in $A^n$.

By a similar argument, the scheme-theoretic image $\mathcal{X}$ of the morphism $\psi : \mu^{-1}(\M) \to (\hat{V \otimes_\C A}) \times_\C \hat{V}$ comes with a natural map $\rho : \mathcal{X} \to \hat{V \otimes_\C A}$ with the property that $\rho^{-1}(q)$ is precisely the scheme-theoretic image of $\mu(q_K)$ under the fiber product map $\pi : \hat{V}_K \to \hat{V}$. Therefore the fiber of $\mathcal{X}$ over a point belonging to $\mathcal{B} \subseteq \hat{V \otimes_\C A}$ is the extension subvariety associated to the corresponding point of $\M$. We can therefore view $\mathcal{Y} = \rho^{-1}(\mathcal{B}) \to \mathcal{B}$ as the family of extension subvarieties indexed by closed points in $\mathcal{M} \cap A^n$.

The rest of this section comprises examples of this technique used to construct families of extension subvarieties. Throughout, we use $\vspan_\C(k_1, \dots, k_r)$ to mean the $\C$-sub-vector space of $K$ spanned by $k_1, \dots, k_r$.

\begin{example}
    Consider the action of $G = \SL_2(\C)$ on $V = M_{2 \times 2}(\C)$. Let $A = \vspan_\C(\sqrt{\Delta})$ and denote by $x_{ij}$ the basis of $(V \otimes_\C A)^*$ dual to $E_{ij} \otimes \sqrt{\Delta}$. The morphism $\varphi$ is given on the level of rings by the map
    \[\C[x_{ij}, \sqrt{\Delta}] \to \frac{K[x_{ij}]}{(\Delta x_{11} x_{22} - \Delta x_{12} x_{21} - \Delta)},\]
    which has kernel $(x_{11} x_{22} - x_{12} x_{21} - 1)$ as $\Delta$ is a unit in the target ring. Thus
    \[\mathcal{A} = \Spec \frac{\C[x_{ij}, \sqrt{\Delta}]}{(x_{11} x_{22} - x_{12} x_{21} - 1)}.\]
    Since the fibers of $\mathcal{A} \to \hat{V \otimes_\C A}$ are all of dimension $1$, we then have
    \[\mathcal{B} =  \Spec \frac{\C[x_{ij}]}{(x_{11} x_{22} - x_{12} x_{21} - 1)}\]
    which encodes the fact that a matrix
    \[\begin{bmatrix}
        x_{11} \sqrt{\Delta} & x_{12} \sqrt{\Delta} \\
        x_{21} \sqrt{\Delta} & x_{22} \sqrt{\Delta}
    \end{bmatrix} \in M_{2 \times 2}(K)\]
    belongs to $\M$ iff $x_{11} x_{22} - x_{12} x_{21} = 1$. For the total space of the family, the morphism $\mu^{-1}(\M) \to (\hat{V \otimes_\C A}) \times_\C \hat{V}$ is given by the ring homomorphism
    \[\C[x_{ij}, z_{ij}] \to \frac{K[x_{ij}]}{(\Delta x_{11} x_{22} - \Delta x_{12} x_{21} - \Delta)}\]
    which sends $z_{ij}$ to $\sqrt{\Delta} x_{ij}$. The kernel is $(x_{ij} z_{k\ell} - x_{k \ell} z_{ij})$, and thus the family of extension subvarieties parametrized by $\M \cap A^4$ is
    \[\begin{array}{ccc}
        \mathcal{Y} & = & \{x_{11} x_{22} - x_{12} x_{21} = 1,\ x_{ij} z_{k \ell} = x_{k \ell} z_{ij}\} \\
        \downarrow & & \downarrow \\
        \mathcal{B} & = & \{x_{11} x_{22} - x_{12} x_{21} = 1\}.
    \end{array}\]
    From this, it's easy to see that the fiber of $\mathcal{Y} \to \mathcal{B}$ over the point $\begin{bsmallmatrix}
        x_{11} & x_{12} \\
        x_{21} & x_{22}
    \end{bsmallmatrix} = \begin{bsmallmatrix}
        c_{11} & c_{12} \\
        c_{21} & c_{22}
    \end{bsmallmatrix} \in M_{2 \times 2}(\C)$ is the Cartan subspace
    \[\left\{ \begin{bmatrix}
        c_{11} x & c_{12} x \\
        c_{21} x & c_{22} x
    \end{bmatrix} {\ \Bigg |\ } x \in \C \right\}\]
    and $\mathcal{Y} \to \mathcal{B}$ is the family of all Cartan subspaces of $M_{2 \times 2}(\C)$ under the action of $\SL_2(\C)$.
\end{example}

\begin{example}
    Again take the action of $G = \SL_2(\C)$ on $V = M_{2 \times 2}(\C)$, now with $A = \vspan_\C(1, \sqrt{\Delta}, \Delta)$. Notice that the function fields of extension subvarieties corresponding to points in $\M$ with coordinates in $A$ embed into $K = \bar{\C(\Delta)}$ either as $\C(\Delta)$ or $\C(\sqrt{\Delta})$. Let $a_{ij}$, $b_{ij}$, and $c_{ij}$ be the basis of $(V \otimes_\C A)^*$ dual to $E_{ij} \otimes 1$, $E_{ij} \otimes \sqrt{\Delta}$, and $E_{ij} \otimes \Delta$ respectively. Using a computation similar to the previous example, we can find that $\mathcal{B}$ is cut out by the equations
    \[a_{11} a_{22} - a_{12} a_{21} = b_{11} a_{22} + a_{11} b_{22} - a_{12} b_{21} - b_{12} a_{21} = b_{11} c_{22} + c_{11} b_{22} - b_{12} c_{21} - c_{12} b_{21} = c_{11} c_{22} - c_{12} c_{21} = 0,\]
    \[a_{11} c_{22} + b_{11} b_{22} + c_{11} a_{22} - a_{12} c_{21} - b_{12} b_{21} - c_{12} a_{21} = 1.\]
    Thus $\mathcal{B}$ is a seven-dimensional closed subvariety of the twelve-dimensional vector space $V \otimes_\C A$. The subvariety $\mathcal{Y} \subset \mathcal{B} \times_\C V$ is given by the equations
    \[(z_{ij} - a_{ij}) b_{k \ell} - (z_{k \ell} - a_{k \ell}) b_{ij} = (z_{11} z_{22} - z_{12} z_{21}) (c_{ij} b_{k \ell} - c_{k \ell} b_{ij})\]
    where $a_{ij}, b_{ij}, c_{ij}$ are coordinates on $\mathcal{B}$ as before and $z_{ij}$ are the coordinates on $V$. Notice that when all $a_{ij} = c_{ij} = 0$, the result of the previous example is recovered and the fibers are linear subspaces of $V$. On the other hand, for generic closed points of $\mathcal{B}$, the fiber is a degree-two curve in $V$. If all the $b_{ij}$ are zero, then the fiber $Y$ has $\C(Y) = \C(V)^G$, while if any $b_{ij}$ is nonzero, then $\C(Y) = \C(\sqrt{\Delta})$, a quadratic extension of $\C(V)^G$. In other words, the quadratic extension $\C(\sqrt{\Delta})$ is the ``generic" function field of the fibers of $\mathcal{Y} \to \mathcal{B}$.
\end{example}

\begin{example}
    Next, consider the action of $G = \C^\times$ on $V = \C^2$ where $c \cdot (z_1, z_2) = (cz_1, c^{-1} z_2)$. Then $\C[V]^G = \C[z_1 z_2]$, and taking $K = \bar{\C(T)}$, we get $\M = \Spec \frac{K[z_1, z_2]}{(z_1 z_2 - T)}$. Let $A = \vspan_\C(T^{m/n}, T^{n-m/n})$ where $\gcd(m, n) = 1$, and denote by $a$, $b$, $c$, and $d$ the basis of $(V \otimes_\C A)^*$ dual to the basis $e_1 \otimes T^{m/n}$, $e_1 \otimes T^{n-m/n}$, $e_2 \otimes T^{m/n}$, and $e_2 \otimes T^{n-m/n}$. The morphism $\varphi$ is given on the level of rings by the homomorphism
    \[\C[a, b, c, d, T^{m/n}, T^{n-m/m}] \to \frac{K[a, b, c, d]}{(ac T^{2m/n} + bd T^{2n-2m/n} + (ad+bc-1) T)},\]
    whose kernel is $(ac (T^{m/n})^2 + bd (T^{n-m/n})^2 + (ad + bc - 1) T^{m/n} T^{n-m/n})$, and therefore
    \[\mathcal{A} = \Spec \frac{\C[a, b, c, d, T^{m/n}, T^{n-m/m}]}{(ac (T^{m/n})^2 + bd (T^{n-m/n})^2 + (ad + bc - 1) T^{m/n} T^{n-m/n})}.\]
    Thus, the locus in $\hat{V \otimes_\C A}$ where the fibers of $\mathcal{A} \to \hat{V \otimes_\C A}$ are full-dimension is
    \[\mathcal{B} = \Spec \frac{\C[a, b, c, d]}{(ac, bd, ad+bc-1)}.\]
    This is a disconnected scheme with two irreducible components
    \[\Spec \frac{\C[a, b, c, d]}{(a, d, bc-1)} \quad \text{and} \quad \Spec \frac{\C[a, b, c, d]}{(b, c, ad-1)}.\]
    To compute $\mathcal{Y}$, the morphism $\psi$ is given on the level of rings as
    \[\C[a, b, c, d, x, y] \to \frac{K[a, b, c, d]}{(ac T^{2m/n} + bd T^{2n-2m/n} + (ad+bc-1) T)}\]
    where $x \mapsto a T^{m/n} + b T^{n-m/n}$ and $y \mapsto c T^{m/n} + d T^{n-m/n}$. By observing that $ay - cx \mapsto (ad-bc) T^{n-m/n}$ and $dx - by \mapsto (ad-bc) T^{m/n}$ under this homomorphism, we conclude that the kernel of this homomorphism is generated by $(ad-bc)^{n-m} (ay-cx)^m - (ad-bc)^m (dx-by)^{n-m}$ and therefore
    \[\mathcal{Y} = \Spec \frac{\C[a, b, c, d, x, y]}{(ac, bd, ad+bc-1, (ad-bc)^{n-m} (ay-cx)^m - (ad-bc)^m (dx-by)^{n-m})}.\]
    For example, given $t \in \C^\times$, the fiber of $\mathcal{Y} \to \mathcal{B}$ over the closed point $(a,b,c,d) = (t, 0, 0, t^{-1}) \in \B$ is
    \[\Spec \frac{\C[x, y]}{(t^{m-n} x^{n-m} - t^m y^m)},\]
    or in other words the subvariety of $V$ cut out by the function $t^{m-n} x^{n-m} - t^m y^m$. This is exactly the extension subvariety associated to the closed point $(t T^{m/n}, t^{-1} T^{n-m/n}) \in \M$, as this function generates the ideal $(x - t T^{m/n}, y - t^{-1} T^{n-m/n}) \cap \C[V]$.
\end{example}

These examples raise interesting questions about deformations of extension subvarieties, which will be the subject of a forthcoming article. Namely, in some fixed family $\mathcal{Y} \to \mathcal{B}$ one can ask questions such as the following:
\begin{itemize}
    \item What is the locus in $\mathcal{B}$ above which the fibers all have the Galois restriction property? Under what conditions does the generic fiber have this property?
    \item Fix a Galois extension $F$ of $\C(V)^G$ inside of $K$. Can we describe the locus in $\mathcal{B}$ above which the canonical images inside $K$ of the fibers' function fields are equal to $F$? Is one option of $F$ generic in the family?
    \item Perform this construction for each $\C$-sub-vector space $A$ of $K$ to get a family $\mathcal{Y}_A \to \mathcal{B}_A$. Taking the direct limit of $\mathcal{Y}_A \to \mathcal{B}_A$ over all such $A$, we obtain a family of ind-constructible sets over $\C$ which is some sense the ``universal" family of extension subvarieties of $V$. What properties does this universal family enjoy?
    \item The universal family in the previous bullet comes equipped with an action of the absolute Galois group $\Gal(K / \C(V)^G)$. Can we say anything interesting about this action?
\end{itemize}

\section{Period integrals from the Galois restriction property} \label{sec:period-integrals}

In this section, we will demonstrate that restriction properties, especially the Galois restriction property on the level of function fields, are quite useful for understanding the period integrals of Calabi-Yau (CY) families. Specifically, we will leverage the Galois restriction property to write explicit formulas for the periods of CY families in terms of invariant functions.

In \cite{Lian_2012}, Lian and Yau studied the period integrals of CY varieties $Y$ given as complete intersections in an ambient manifold $X$. Naively, these period integrals are defined by
\[\Pi = \int_\gamma \omega_Y\]
where $\gamma \in H_{\dim Y}(Y, \Z)$ and $\omega_Y$ is a holomorphic volume form on $Y$, which is unique up to scalar thanks to the fact that $Y$ is CY.

The study of these period integrals is rather subtle: the parameter space of such complete intersections is a Zariski open set in a representation $V$ of a reductive group $G$ and is generally noncompact and noncontractible, so a canonical choice of volume form must be made on each complete intersection. Lian and Yau's contribution was twofold: one, to provide a canonical normalization through a Poincar\'e residue technique on a special type of $G$-equivariant principal bundle over $X$; and two, to describe a $\D$-module, called a tautological system, which governs the period integrals as functions on the moduli space. The construction can also be adapted to so-called ``fractional complete intersections", which are CY branched covers of a fixed $X$.

For simplicity, let $X$ be a compact complex manifold whose anticanonical bundle $-K_X$ is very ample. Consider the family of CY hypersurfaces given by the zero locus of sections of $-K_X$, and let $V = H^0(-K_X)^*$. We are interested in the locus of sections $\sigma$ so that $\{\sigma = 0\}$ is a smooth hypersurface; this forms a Zariski open set $\mathcal{B}$ in $V^*$ on which the tautological system is defined. There is an action of $G = \Aut(X) \times \C^\times$ on $\mathcal{B}$ induced by automorphisms of $X$ and the $\C^\times$-action on scaling sections of $-K_X$. 

From this data, the tautological system is a $\D$-module comprising two types of operators: embedding polynomial operators encoding the geometry of the embedding of the ambient space $X$ into $\P V$, and symmetry operators encoding the transformation properties of the periods under the action of $G$. In many cases, such as the case of hypersurfaces in $\P^n$ or other homogeneous spaces, $\Aut(X)$ is semisimple and the transformation properties encoded by the symmetry operators are quite close to $G$-invariance, typically taking the form of invariance under the action of $\Aut(X)$ and homogeneity under the action of $\C^\times$. Our objective is to combine this observation and the theory of extension subvarieties and their restriction properties, in order to lift solutions to the tautological system defined on subvarieties of $\mathcal{B}$ to the entire moduli space. Although there are many descriptions of period integrals which can be viewed as functions on certain subvarieties of $\mathcal{B}$, we are not aware of any procedure for finding explicit solutions to tautological systems in terms of $G$-invariants other than the method we are about to introduce. The technique relies on interplay between the algebraic and analytic categories, as explained in the two following lemmas. 

Let $V$ be any representation of a reductive group $G$. (In particular, $\C(V)^G$ may be strictly larger than $\Frac(\C[V]^G)$.) By Rosenlicht's theorem \cite{rosenlicht}, there is a variety $Z$ whose function field is $\C(V)^G$ and a dominant rational map $\pi : V \dashto Z$ whose fibers are $G$-orbits in $V$.

\begin{lemma}[Analytic mapping lemma] \label{lem:analytic-mapping}
    Let ($Y, W)$ have the Galois restriction property. Then there is a nonempty Zariski open set $U$ in $V$ all of whose points $p$ satisfy the condition that $G \cdot p$ intersects $Y$. Furthermore, for any such $p$ belonging to this open set and any $y \in (G \cdot p) \cap Y$, there is an analytic neighborhood $U_p$ of $p$ in $V$, an analytic neighborhood $U_y$ of $y$ in $Y$, and a holomorphic map $\rho : U_p \to U_y$ so that $\pi \circ \rho = \pi$. The fibers of $\rho$ are intersections of $G$-orbits with $U_p$.
\end{lemma}
\begin{proof}
    Again by Rosenlicht's theorem, there is a variety whose function field is $\C(Y)^W$ and a dominant rational map from $Y$ to this variety whose fibers are $W$-orbits in $Y$; since $\C(V)^G = \C(Y)^W$ by hypothesis, we may as well take this variety to be $Z$ and the map $Y \dashto Z$ to be $\pi|_Y$. Since $\pi$ and $\pi|_Y$ are dominant, their images each contain a dense Zariski open subset of $Z$. Let $U'$ be the intersection of these dense open sets minus the branch locus of the cover $Y \dashto Z$. Then $U = \pi^{-1}(U')$ is a nonempty Zariski open set in $V$, and by construction the $G$-orbit of each point in $U$ intersects $Y$, proving the first claim. For the second claim, take any such $p$ and $y$. Since $Y \dashto Z$ is a Galois cover and thus \'etale away from its branch locus, there is an analytic neighborhood $U_y$ of $y$ for which $\pi|_{U_y}$ is a local biholomorphism into $Z$. Taking $U_p = \pi^{-1}(\pi(U_y))$ and $\rho = \pi|_{U_y}^{-1} \circ \pi$, we obtain the second claim. The final statement is immediate because the fibers of $\pi$ are $G$-orbits.
\end{proof}

\begin{lemma}[Analytic lifting lemma] \label{lem:analytic-lifting}
    Let $V$ be a representation of reductive $G$ and let $(Y, W)$ have the Galois restriction property. Let $p$ and $y$ be as in the previous lemma. Then for any holomorphic function $\varphi$ defined on an analytic neighborhood of $y$ in $Y$, there is a  holomorphic function $\tilde{\varphi}$ defined on an analytic neighborhood $U_p$ of $p$ in $V$ with the property that $\tilde{\varphi}$ is constant on the intersections of $G$-orbits with $U_p$ and factors through $\varphi$.
\end{lemma}
\begin{proof}
    Let $U_p$ and $U_y$ be the analytic neighborhoods of $p$ in $V$ and $y$ in $Y$, and $\rho : U_p \to U_y$ the holomorphic map of the analytic mapping lemma. Then $\tilde{\varphi} = \varphi \circ \rho$ clearly satisfies the desired properties.
\end{proof}

\begin{remark}
    Denoting by $\pi$ the morphism $V \to V \sslash G$, Luna showed in \cite{Luna-fonctions-differentiables} that for any analytic open set $\Omega \subseteq V \sslash G$, the pullback map $\pi^*$ from the ring of analytic functions on $\Omega$ to the ring of $G$-invariant analytic functions on $\pi^{-1}(\Omega) \subseteq V$ is an isomorphism. Our Lemmas \ref{lem:analytic-mapping} and \ref{lem:analytic-lifting} are similar in spirit, but seem to improve upon Luna's result in two ways. First, our result allows for the finite group $W$ to enter the picture whereas Luna considered functions factoring through $V \sslash G$ on the nose. The second difference is that our lemmas apply to analytic functions defined on a much wider class of analytic open sets in $V$, as shown in the following example. This is because Luna's theorem breaks down in the presence of unstable orbits, and thus functions on $V \sslash G$ are generally quite far from capturing all the $G$-invariant functions on $V$. The invariant field is a much better object for parametrizing $G$-orbits, and working on the level of the Galois restriction property avoids such issues.
\end{remark}

\begin{example}
    Consider the action of $G = \C^\times$ on $V = \C^2$ where $c \cdot (z_1, z_2) = (c z_1, c z_2)$. Here every orbit is unstable, $V \sslash G$ is a point, and Luna's theorem applies only to $G$-invariant analytic functions defined everywhere on $V$. Analyzing the orbit structure of $V$, one sees that the only such functions are constant. However, there are plenty of nonconstant $G$-invariant analytic functions defined on open sets in $V$ whch can be captured by Lemmas \ref{lem:analytic-mapping} and \ref{lem:analytic-lifting}. For example, consider the subvariety $Y = \{z_2 = 1\}$ with the Galois restriction property. Any locally defined $G$-invariant analytic function on $Y$ can be lifted to an analytic function on any simply connected analytic open set $U \subset V$ which does not contain the $z_1$-axis by simply precomposing with the map $\rho : U \to Y$, $(z_1, z_2) \mapsto (\frac{z_1}{z_2}, 1)$.
\end{example}

To deduce our invariant-theoretic period formulas, we will use Lemmas \ref{lem:analytic-mapping} and \ref{lem:analytic-lifting} to build a function on $V$ with prescribed values on a subvariety and transformation properties under the $\hat{G}$-action, then utilize the following lemma to argue that it must in fact be a period.
\begin{lemma}[Uniqueness lemma] \label{lem:uniqueness-lemma}
    Fix a tautological system $\tau$ on a Zariski open set $U$ of a representation $V$ of a connected reductive group $G$, whose symmetry operators encode equivariance properties under the $G$-action on $V$. Let $Y$ be a subvariety of $U$ such that the set $G \cdot Y$ comprising $G$-translates of $Y$ is dense in $U$. Suppose $\varphi$ and $\psi$ are two analytic functions on $U$, both annihilated by the symmetry operators of $\tau$, whose values agree on $Y$. Then $\varphi = \psi$ everywhere on $U$.
\end{lemma}
\begin{proof}
    The symmetry operators dictate how a function locally transforms along orbits under the action of $G$. Since $G$ is connected, so are the $G$-orbits in $U$. Now, because $\varphi$ is annihilated by the symmetry operators, the value of $\varphi$ at a point $p$ therefore determines its values along the entire orbit $G \cdot p$. As $G \cdot Y$ is dense in $U$, the claim follows.
\end{proof}

Throughout the discussion of the following period integral problems, we will make extensive use of work by Matsumoto, Sasaki, and Yoshida regarding reduction of hypergeometric systems to a particular closed subvariety. This result is really quite remarkable, because in general the question of restricting PDE systems or $\D$-modules to a subvariety is quite complex. We anticipate that the Galois restriction property plays a crucial role in this reduction process, and we plan to investigate it more deeply in the future.

\subsection{Double covers of $\P^{n-1}$} \label{sec:double-covers-Pn}

We begin by considering the family of CY double covers of $\P^{n-1}$ branched along $2n$ hyperplanes in general position. Such a double cover is specified by a generic choice of $2n$ sections of $\O(1)$, or by a $n \times 2n$ complex matrix $Z = [z_{ij}]$ with no zero minors; call the space of such matrices by $M_{n \times 2n}^o(\C)$. This is a Zariski open subvariety of the representation $M_{n \times 2n}(\C)$ of $G = \SL_n(\C) \times (\C^\times)^{2n}$, where $\SL_n(\C)$ acts by left multiplication on $Z$ and the $i$-th copy of $\C^\times$ scales the $i$-th column of $Z$. Geometrically, the $\SL_n(\C)$ action is the action on $\bigoplus_{i=1}^{2n} H^0(\O(1))$ induced by automorphisms of $\P^{n-1}$, while the torus action simply rescales the sections defining the branch locus, whence it is easy to see that the group action identifies isomorphic double covers and preserves $M_{n \times 2n}^o(\C)$.

Adapting the Poincar\'e residue procedure defined by Lian and Yau for Calabi-Yau complete intersections \cite{Lian_2012} to the case of branched covers, one obtains the canonical volume form
\[\omega(Z) = \frac{\sum_{i=1}^n (-1)^i w_i dw_1 \wedge \cdots \wedge \hat{dw_i} \wedge \cdots \wedge dw_n}{\sqrt{\prod_{k=1}^{2n} z_{1,k} w_1 + \cdots + z_{n, k} w_n}}\]
on each double cover, where $w_1, \dots, w_n$ are the homogeneous coordinates on $\P^{n-1}$ and the hat denotes omission. The resulting period $\Pi(Z) = \int_\gamma \omega(Z)$ is a solution to the tautological system generated by the operators
\[\begin{cases}
    \dee_{z_{ip}} \dee_{z_{jq}} - \dee_{z_{iq}} \dee_{z_{jp}} & 1 \leq i < j \leq n,\ 1 \leq p < q \leq 2n \\
    \sum_{j=1}^{2n} z_{ij} \dee_{z_{mj}} + \delta_{im} & 1 \leq i, m \leq n \\
    \sum_{i=1}^n z_{ij} \dee_{z_{ij}} + \frac{1}{2} & 1 \leq j \leq 2n.
\end{cases}\]
The operators on the second line express the $\SL_n(\C)$-invariance of the periods, while the operators on the third line express the fact that the periods are homogeneous of degree $-\frac{1}{2}$ with regard to each $\C^\times$ action, as can be seen from the expression for $\omega(Z)$.

The tautological system for this family of double covers is a special case of the Aomoto-Gelfand hypergeometric system $E(k, n)$ studied by Matsumoto, Sasaki, and Yoshida in \cite{MSY-aomoto-gelfand-3-6} for the case $k = 3$, $n = 6$ and \cite{MSY-monodromy} for general $k$ and $n$. This is a PDE system on $M_{k \times n}^o(\C)$ with an action by $G = \SL_k(\C) \times (\C^\times)^n$. In the latter article, they proved three key statements about this system and the geometry of the parameter space:
\begin{enumerate}
    \item The subvariety $Y \subset M_{k \times n}^o(\C)$ consisting of matrices of the form
    \[\begin{bmatrix}
		1 & & & & 1 & 1 & \hdots & 1 \\
		& 1 & & & 1 & x_{2,k+2} & \hdots & x_{2,n} \\
		& & \ddots & & \vdots & \vdots & \ddots & \vdots \\
		& & & 1 & 1 &  x_{k,k+2} & \hdots & x_{k,n}
	\end{bmatrix}\]
    intersects each $G$-orbit exactly once. In other words, the set $G \cdot Y$ comprising $G$-translates of $Y$ is all of $M_{k \times n}^o(\C)$.
    \item The system $E(k, n)$ has a reduction to a PDE system $E'(k, n)$ on $Y$ generated by second-order operators, in the sense that the restriction of a solution of $E(k, n)$ to $Y$ is a solution of $E'(k, n)$.
    \item There is an explicit power series solution of $E'(k, n)$, unique up to scalar, which is holomorphic in a neighborhood of
    \[\begin{bmatrix}
        1 & & & & 1 & 1 & \hdots & 1 \\
        & 1 & & & 1 & 0 & \hdots & 0 \\
        & & \ddots & & \vdots & \vdots & \ddots & \vdots \\
        & & & 1 & 1 & 0 & \hdots & 0
    \end{bmatrix}\]
    in $Y$ and extends to a multivalued meromorphic function on $Y$.
\end{enumerate}
(Matsumoto, Sasaki, and Yoshida considered an action by the slightly larger group $\GL_k(\C)\ \times (\C^\times)^n$, but every $M \in \GL_k(\C)$ factors as the product of $(\det M)^{-1/k} M \in \SL_k(\C)$ and $(\det M)^{1/k} I_k$, the latter of which acts on $M_{k \times n}^o(\C)$ identically to $((\det M)^{1/k}, \dots, (\det M)^{1/k}) \in (\C^\times)^n$.) Our goal will be to lift the explicit solution of $E'(n, 2n)$ from $Y$ to a solution of $E(n, 2n)$ on all of $M_{n \times 2n}^o(\C)$. Because we will work with rational functions, the restriction to a Zariski open makes no difference, and we can consider rational functions on all of $M_{n \times 2n}(\C)$. To do so, we must first understand the invariant theory of $M_{n \times 2n}(\C)$ as a representation of $G$.

\begin{lemma} \label{lem:invariant-rational-functions-as-cross-ratios}
    Let $f \in \C(M_{n \times 2n}(\C))^{G}$. Then $f$ can be expressed as a rational function of ``cross-ratios" of the form
	\[\frac{\Delta_{i_1 \cdots i_{n-2} j_1 j_2} \Delta_{i_1 \cdots i_{n-2} \ell_1 \ell_2}}{\Delta_{i_1 \cdots i_{n-2} j_1 \ell_2} \Delta_{i_1 \cdots i_{n-2} \ell_1 j_2}}\]
	where $i_1, \dots, i_{n-2}, j_1, j_2, \ell_1, \ell_2$ are distinct integers between $1$ and $2n$ and $\Delta_{i_1 \cdots i_n}$ denotes the minor of the $i_1, \dots, i_n$ columns of $Z$.
\end{lemma}
\begin{proof}
    Since $\SL_n(\C)$ has no nontrivial characters it follows that $\C(M_{n \times 2n})^{\SL_n(\C)} = \Frac( \C[M_{n \times 2n}]^{\SL_n(\C)} )$, and it is well-known that this invariant ring is generated by the minors of $Z$. Thus, all that remains is to show that the cross-ratios generate the subfield of $\C(M_{n \times 2n})^{\SL_n(\C)}$ comprising functions which are invariant under the torus action. Without loss of generality, we can assume that $f$ is written as a reduced fraction of products of the minors, in which case $(\C^\times)^{2n}$-invariance implies that each subscript between $1$ and $2n$ appears the same number of times on the numerator and denominator of $f$. 
    
    Multiplying $f$ by a suitable power of 
	\[\frac{\Delta_{i_1 \cdots i_{n-2} 1 (2n)} \Delta_{i_1 \cdots i_{n-2} \ell_1 \ell_2}}{\Delta_{i_1 \cdots i_{n-2} 1 \ell_2} \Delta_{i_1 \cdots i_{n-2} \ell_1 (2n)}},\]
	we can ensure that $\Delta_{i_1 \cdots i_{n-2} 1 (2n)}$ does not appear in the resulting product when written in reduced form. Continuing in this way and replacing $2n$ by $2n-1, \dots, n+1$, the result is that $\Delta_{1 \cdots n}$ is the only minor with subscript $1$ appearing in the product of $f$ and cross-ratios when written in reduced form. But $\Delta_{1 \cdots n}$ is the only minor with subscript $1$ which could appear in the resulting product, which is still $(\C^\times)^{2n}$-invariant, so the product must contain no instances of the subscript $1$ when written in reduced form. Continuing, we repeat this process to eliminate all instances of the subscripts $2, \dots, n-1$ (which is possible because there are $n+2$ integers between $n-1$ and $2n$).

    Let $g$ denote the resulting product of $f$ with the cross-ratios identified in the previous paragraph, so that $g$ is a ratio of products of the minors $\Delta_{n \cdots \hat{i} \cdots (2n)}$ (here a hat on a subscript indicates omission). We will show that $g$ is a complex constant, from which the claim immediately follows. For each $j$ between $n$ and $2n$, invariance under the $j$-th $\C^\times$ action implies that
	\[\sum_{\ell_1, \dots, \ell_{n-1}} p_{j \ell_1 \cdots \ell_{n-1}} = 0\]
    where $p_{i_1 \cdots i_n}$ is the power of $\Delta_{i_1 \cdots i_n}$ appearing in the fraction $g$, and the sum ranges over distinct $\ell_1, \dots, \ell_{n-1}$ between $n$ and $2n$ and not equal to $j$. The difference of the equations corresponding to the indices $j = k$ and $j = \ell$ reads
	\[p_{n \cdots k \cdots \hat{\ell} \cdots (2n)} - p_{n \cdots \hat{k} \cdots \ell \cdots (2n)} = 0,\]
	implying that all the powers $p_{j \ell_1 \cdots \ell_{n-1}}$ are equal. This is only possible if they are all zero, which proves that $g$ is a constant.
\end{proof}

\begin{lemma} \label{lem:G-invt-monomials-gens}
	$\C(M_{n \times 2n}(\C))^{G}$ is generated over $\C$ by the cross-ratios 
	\[f_{ij}(Z) = \frac{\Delta_{1 \cdots \hat{i} \cdots n j} \Delta_{2 \cdots n (n+1)}}{\Delta_{2 \cdots n j} \Delta_{1 \cdots \hat{i} \cdots n (n+1)}},\]
	where $i$ is an integer between $2$ and $n$, and $j$ is an integer between $n+2$ and $2n$.
\end{lemma}
\begin{proof}
    Let $f(Z)$ be a $G$-invariant rational function, and consider its restriction $f|_Y$, which is well-defined because the cross-ratios appearing in Lemma \ref{lem:invariant-rational-functions-as-cross-ratios} each restrict to a rational function of the coordinates of $Y$. Then $f|_Y$ can be expressed as a rational function $g(X)$ of the coordinate functions $x_{ij}$ on $Y$, and since $f_{ij}$ restricts to $x_{ij}$ on $Y$, we have the equation $f|_Y(X) = g(f_{ij}|_Y(X))$. Since $f$ and each $f_{ij}$ are $G$-invariant and $Y$ contains a point in every $G$-orbit, it follows that $f(Z) = g(f_{ij}(Z))$ on all of $M_{n \times 2n}(\C)$.
\end{proof}

Note that in the case $n = 2$, there is only one cross-ratio as defined in the Lemma, $f_{24}$, which is exactly the familiar cross-ratio of four points on $\P^1$. This cross-ratio is well-known as the only projective invariant of four points. In the language of invariant theory, Lemma \ref{lem:G-invt-monomials-gens} is saying that $Y$ is a subvariety of $M_{n \times 2n}(\C)$ with the Galois restriction property, and $\C(Y) / \C(M_{n \times 2n}(\C))^{G}$ is the trivial extension. For our purposes, this means that the map $\rho$ defined in the analytic mapping lemma, Lemma \ref{lem:analytic-mapping}, is single-valued and well-defined on the entire Zariski open $M_{n \times 2n}^o$, sending
\[Z \mapsto \begin{bmatrix}
    1 & & & & 1 & 1 & \hdots & 1 \\
    & 1 & & & 1 & f_{2,n+2}(Z) & \hdots & f_{2,2n}(Z) \\
    & & \ddots & & \vdots & \vdots & \ddots & \vdots \\
    & & & 1 & 1 &  f_{n,n+2}(Z) & \hdots & f_{n,2n}(Z)
\end{bmatrix}.\]
Now we will leverage this map to apply the analytic lifting lemma, Lemma \ref{lem:analytic-lifting}, in order to produce a period formula on all of $M_{n \times 2n}^o$.

\begin{theorem}[Period formula for CY double covers of $\P^{n-1}$] \label{thm:periods-double-covers}
    The holomorphic period of the family of CY double covers of $\P^{n-1}$ is given on an analytic open set in $M_{n \times 2n}^o(\C)$ in terms of $\SL_n(\C)$-invariant functions by the power series
    \[\Pi(Z) = \left( \frac{\left( \prod_{i=n+2}^{2n} \Delta_{2 \cdots n i} \right) \left( \prod_{\{j_1, \dots, j_{n-2}\} \subseteq \{2, \dots, n\}} \Delta_{1 j_1 \cdots j_{n-2} (n+1)} \right) }{\left( \Delta_{1 2 \cdots n} \Delta_{2 \cdots n (n+1)} \right)^{n-2}} \right)^{-\frac{1}{2}} \sum_\ell A(\ell) f_{ij}(Z)^\ell,\]
    where $\ell = (\ell_{ij})$ ranges over $\Z_{\geq 0}$-valued multiindices with $2 \leq i \leq n$ and $n+2 \leq j \leq 2n$,
    \[A(\ell) = \frac{\prod_{j=n+2}^{2n} (\frac{1}{2}, \sum_{i=2}^n \ell_{ij}) \prod_{i=2}^n (\frac{1}{2}, \sum_{j=n+2}^{2n} \ell_{ij})}{(\frac{n}{2}, \sum_{i=2}^n \sum_{j=n+2}^{2n} \ell_{ij}) \prod_{i=2}^n \prod_{j=n+2}^{2n} (1, \ell_{ij})},\]
    and $(a, n) = \frac{\Gamma(a+n)}{\Gamma(a)}$. This function extends to a multivalued meromorphic function on $M_{n \times 2n}^o(\C)$ by analytic continuation.
\end{theorem}
\begin{proof}
    The explicit solution of $E'(n, 2n)$ from \cite{MSY-monodromy} is a multivalued function defined in a neighborhood $U_\mathbf{0}$ of the point
    \[\mathbf{0} = \begin{bmatrix}
        1 & & & & 1 & 1 & \hdots & 1 \\
        & 1 & & & 1 & 0 & \hdots & 0 \\
        & & \ddots & & \vdots & \vdots & \ddots & \vdots \\
        & & & 1 & 1 & 0 & \hdots & 0
    \end{bmatrix}\]
    in $Y$ by $\varphi(X) = \sum_\ell A(\ell) X^\ell$ where $\ell = (\ell_{ij})$ and $A(\ell)$ are as in the statement of the theorem, and extended to the rest of $Y$ by analytic continuation. By combining the explicit form of $\rho$ with Lemma \ref{lem:analytic-lifting}, we see that $\varphi(f_{ij}(Z))$ is a $G$-invariant function on $M_{n \times 2n}^o(\C)$ which restricts to $\varphi(X)$ on $Y$. In particular, $\varphi(f_{ij}(Z))$ is given in $(G \cdot U_{\mathbf{0}}) \cap M_{n \times 2n}^o(\C)$ by the power series $\sum_\ell A(\ell) f_{ij}(Z)^\ell$ in the $f_{ij}(Z)$. However, while $\Pi(Z)$ is expected to be $\SL_n(\C)$-invariant, as described above, it is expected to be homogeneous of degree $-\frac{1}{2}$ in each $\C^\times$ action. 

    This discrepancy can be remedied by the introduction of a prefactor which is $\SL_n(\C)$-invariant, homogeneous of degree $-\frac{1}{2}$ in each $\C^\times$ action, and which is constant on $Y$ so as to preserve the restriction properties of $\varphi(f_{ij}(Z))$. We set
    \[P(Z) = \left( \frac{\left( \prod_{i=n+2}^{2n} \Delta_{2 \cdots n i} \right) \left( \prod_{\{j_1, \dots, j_{n-2}\} \subseteq \{2, \dots, n\}} \Delta_{1 j_1 \cdots j_{n-2} (n+1)} \right) }{\left( \Delta_{1 2 \cdots n} \Delta_{2 \cdots n (n+1)} \right)^{n-2}} \right)^{-\frac{1}{2}}.\]
    $P(Z)$ is clearly $\SL_n(\C)$-invariant, and each of the minors appearing in this expression takes value $\pm 1$ on $Y$. By counting the number of times each column index appears in the numerator and denominator, we can see that it is homogeneous of degree $-\frac{1}{2}$ with respect to the action of each $\C^\times$. By construction, the resulting expression
    \[\Pi(Z) = P(Z) \varphi(f_{ij}(Z))\]
    is annihilated by the symmetry operators of the tautological system $E(n, 2n)$ and restricts to $\varphi$ on $Y$ (perhaps up to scalar). The same is true of the holomorphic period of this family thanks to the reduction result of Matsumoto, Sasaki, and Yoshida. By Lemma \ref{lem:uniqueness-lemma}, the function $\Pi$ therefore is exactly the holomorphic period of this double cover family.
\end{proof}

\subsubsection{Example: elliptic curves as double covers of $\P^1$} \label{sec:elliptic-curves-double-covers}
In the case $n = 2$, the family of elliptic curves given as CY double covers of $\P^1$ is called the Legendre family. In fact, every elliptic curve arises in this way. Beginning with the action of $G = \SL_2(\C) \times (\C^\times)^4$ on the parameter space $M_{2 \times 4}^o(\C)$, it is well-known that the classical cross-ratio of four points in $\P^1$, given in terms of the matrix entries by $\frac{\Delta_{14} \Delta_{23}}{\Delta_{13} \Delta_{24}}$, generates $\C(M_{2 \times 4}(\C))^{G}$. From this, it is clear to see that the subvariety $Y = \{\begin{bsmallmatrix}
	1 & 0 & 1 & 1 \\
	0 & 1 & 1 & \lambda
\end{bsmallmatrix}\}$ has the Galois restriction property with trivial Galois group, because the restriction of the cross-ratio to $Y$ is exactly the coordinate function $\lambda$. Geometrically, $Y$ describes the family of elliptic curves cut out by the cubic equation $y^2 = x (x-1) (x - \lambda)$. On $Y$, the reduced Aomoto-Gelfand system $E'(2, 4)$ is given by the single operator
\[\theta^2 - \lambda \left( \theta + \frac{1}{2} \right)^2\]
where $\theta = \lambda \dee_\lambda$. This is the Gauss hypergeometric equation, and the Picard-Fuchs equation of this one-parameter subfamily described in \cite{HLY2019}. The explicit solution to the reduced Aomoto-Gelfand system from \cite{MSY-monodromy} in this case is the Gauss hypergeometric function
\[\varphi(\lambda) = \sum_{n=0}^\infty \frac{\Gamma(n + \frac{1}{2})^2}{\Gamma(\frac{1}{2})^2 \Gamma(n+1)^2} \lambda^n = \leftindex_2{F}_1 \left( \frac{1}{2}, \frac{1}{2}; 1; \lambda \right)\]
which was also recovered in \cite{HLY2019}. The $\SL_2(\C)$-invariant and $(\C^\times)^4$-homogeneous prefactor described above takes the form
\[P(Z) = (\Delta_{13} \Delta_{24})^{-\frac{1}{2}}\]
in this case. Putting the pieces together, we conclude that the holomorphic period of the Legendre family is the function on $M_{2 \times 4}^o(\C)$ defined by
\[\Pi(Z) = (\Delta_{13} \Delta_{24})^{-1/2} \sum_{n=0}^\infty \frac{\Gamma\left( n + \frac{1}{2} \right)^2}{\Gamma\left( \frac{1}{2} \right)^2 \Gamma(n + 1)^2} \left( \frac{\Delta_{14} \Delta_{23}}{\Delta_{13} \Delta_{24}} \right)^n\]
on a neighborhood where this power series converges, and extended to the rest of $M_{2 \times 4}^o(\C)$ by multivalued analytic continuation.

\subsubsection{Example: K3 surfaces of Picard number 16 as double covers of $\P^2$}
In the case $n = 3$, the double covers of $\P^2$ branched along six lines are K3 surfaces with $A_1$ singularities above the intersections of pairs of lines. Blowing up the singularities yields a family of smooth K3 surfaces of Picard number 16. Using the action of $\SL_3(\C) \times (\C^\times)^6$ on $M_{3 \times 6}^o(\C)$, every $3 \times 6$ matrix with nonzero minors can uniquely be put into the form
\[\begin{bmatrix}
    1 & 0 & 0 & 1 & 1 & 1 \\
    0 & 1 & 0 & 1 & x_{25} & x_{26} \\
    0 & 0 & 1 & 1 & x_{35} & x_{36}
\end{bmatrix},\]
a process sometimes referred to as gauge fixing \cite{hosono2019k3}. The subvariety of $M_{3 \times 6}(\C)$ comprising such matrices is the $Y$ identified by Matsumoto, Sasaki, and Yoshida \cite{MSY-monodromy}, and has the Galois restriction property with trivial Galois group. Namely, the coordinates $x_{25}$, $x_{35}$, $x_{26}$, and $x_{36}$ on $Y$ are the restrictions of the $G$-invariant rational functions
\[\frac{\Delta_{135} \Delta_{234}}{\Delta_{235} \Delta_{134}},\ \frac{\Delta_{125} \Delta_{234}}{\Delta_{235} \Delta_{124}},\ \frac{\Delta_{136} \Delta_{234}}{\Delta_{236} \Delta_{134}},\ \text{and}\ \frac{\Delta_{126} \Delta_{234}}{\Delta_{236} \Delta_{124}}\]
on $M_{3 \times 6}(\C)$, respectively. Compactifications of $Y$ were also studied in \cite{HLY2019} in order to identify large complex structure limits of this family. The power series solution to the reduced Aomoto-Gelfand system $E(3, 6)$ is
\[\varphi(X) = \sum_N \frac{\Gamma(n_{25} + n_{35} + \frac{1}{2}) \Gamma(n_{26} + n_{36} + \frac{1}{2}) \Gamma(n_{25} + n_{26} + \frac{1}{2}) \Gamma(n_{35} + n_{36} + \frac{1}{2})}{2 \Gamma(\frac{1}{2})^3 \Gamma(n_{25} + n_{35} + n_{26} + n_{36} + \frac{3}{2}) \Gamma(n_{25} + 1) \Gamma(n_{35} + 1) \Gamma(n_{26} + 1) \Gamma(n_{36} + 1)} x_{25}^{n_{25}} x_{35}^{n_{35}} x_{26}^{n_{26}} x_{36}^{n_{36}}\]
where the sum ranges over $2 \times 2$ matrices $N = \begin{bsmallmatrix}
    n_{25} & n_{26} \\
    n_{35} & n_{36}
\end{bsmallmatrix}$ with entries in $\Z_{\geq 0}$, and the $\SL_3(\C)$-invariant and $(\C^\times)^6$-homogeneous prefactor for this family is
\[P(Z) = \left( \frac{\Delta_{235} \Delta_{236} \Delta_{124} \Delta_{134}}{\Delta_{123} \Delta_{234}} \right)^{-\frac{1}{2}}.\]
Hence, the holomorphic period of this family of K3 surfaces is given explicitly by
\begin{align*}
    \Pi(Z) &= \left( \frac{\Delta_{235} \Delta_{236} \Delta_{124} \Delta_{134}}{\Delta_{123} \Delta_{234}} \right)^{-\frac{1}{2}} \\
    &\quad \times \sum_N {\Bigg [} \frac{\Gamma(n_{25} + n_{35} + \frac{1}{2}) \Gamma(n_{26} + n_{36} + \frac{1}{2}) \Gamma(n_{25} + n_{26} + \frac{1}{2}) \Gamma(n_{35} + n_{36} + \frac{1}{2})}{2 \Gamma(\frac{1}{2})^3 \Gamma(n_{25} + n_{35} + n_{26} + n_{36} + \frac{3}{2}) \Gamma(n_{25} + 1) \Gamma(n_{35} + 1) \Gamma(n_{26} + 1) \Gamma(n_{36} + 1)} \\
    &\qquad \qquad \quad \times \left( \frac{\Delta_{135} \Delta_{234}}{\Delta_{235} \Delta_{134}} \right)^{n_{25}} \left( \frac{\Delta_{125} \Delta_{234}}{\Delta_{235} \Delta_{124}} \right)^{n_{35}} \left( \frac{\Delta_{136} \Delta_{234}}{\Delta_{236} \Delta_{134}} \right)^{n_{26}} \left( \frac{\Delta_{126} \Delta_{234}}{\Delta_{236} \Delta_{124}} \right)^{n_{36}} {\Bigg ]}
\end{align*}
on the domain of convergence of this power series, and is defined on the rest of $M_{3 \times 6}^o(\C)$ by multivalued analytic continuation.

\subsection{Elliptic curves in $\P^2$}
Next, we consider the cubic family of elliptic curves, expressed as smooth hypersurfaces in $\P^2$. These hypersurfaces are given by a nondegenerate section of $\O(3)$, which can be written as a cubic in the homogeneous coordinates $x_1$, $x_2$, $x_3$ of $\P^2$ of the form
\[\sigma = z_{111} x_1^3 + z_{112} x_1^2 x_2 + z_{113} x_1^2 x_3 + z_{122} x_1 x_2^2 + z_{123} x_1 x_2 x_3 + z_{133} x_1 x_3^2 + z_{222} x_2^3 + z_{223} x_2^2 x_3 + z_{233} x_2 x_3^2 + z_{333} x_3^3.\]
The collection of such nondegenerate cubics is a Zariski open set $V^o$ in the representation $V = H^0(\O(3))$ of $G = \GL_3(\C)$, where the action of $G$ is induced by its action on $\P^2$.

In this scenario, the canonical volume form on the hypersurface $\{\sigma = 0\}$ is given in terms of the coordinates $z_{ijk}$ by
\[\omega(Z) = \frac{x_1 dx_2 \wedge dx_3 - x_2 dx_1 \wedge dx_3 + x_3 dx_1 \wedge dx_2}{z_{111} x_1^3 + z_{112} x_1^2 x_2 + \cdots + z_{333} x_3^3}\]
by the Poincar\'e residue method of Lian and Yau \cite{Lian_2012}. Thus the period integrals
\[\Pi(Z) = \int_\gamma \omega(Z)\]
are $\SL_3(\C)$-invariant and homogeneous of degree $-3$ with respect to the diagonal copy of $\C^\times \subset \GL_3(\C)$.

We will once again lift a formula for the period integral on a closed subvariety to a global period formula on the entire parameter space using the Galois restriction property. Remarkably, we will see that the period is a power series involving terms expressible in radicals. The end goal is to leverage what we already know about the Legendre family of elliptic curves from the previous section, but first we must touch upon the invariant theory of the $G$-action on $V$. Since $\SL_3(\C)$ has no nontrivial characters, we have $\C(V)^{\SL_3(\C)} = \Frac(\C[V]^{\SL_3(\C)})$, and it is a classical result that this invariant ring is a polynomial ring generated by the Aronhold invariants $S$ and $T$ of degrees $4$ and $6$ respectively (see Section \ref{sec:introduction}). The $j$-invariant of the elliptic curve corresponding to a nondegenerate cubic in $V$ is given in terms of $S$ and $T$ by the rational function $J = \frac{S^3}{T^2 - 4 S^3}$, which generates $\C(V)^G$ as an extension of $\C$.

Now, it is a well-known fact that every elliptic curve over $\C$ can be put into Weierstrass form $a x_1^3 + b x_1^2 x_3 + c x_1 x_3^2 + d x_3^3 - x_2^2 x_3 = 0$ using the action of $\GL_3(\C)$ introduced above \cite{silverman-elliptic-curves}. From this form, it is then simple to put the elliptic curve into Legendre form $x_1 (x_1 - x_3) (x_1 - \lambda x_3) - x_2^2 x_3 = x_1^3 + (-\lambda - 1) x_1^2 z_3 + \lambda x_1 x_3^2 - x_2^2 x_3 = 0$ by another $\GL_3(\C)$-transformation. We take $Y$ to be the subvariety of $V$ comprising those Legendre-form cubics; explicitly, $Y$ is the one-dimensional subvariety
\[\{z_{111} = 1,\ z_{113} = -\lambda - 1,\ z_{133} = \lambda,\ z_{223} = -1,\ z_{112} = z_{122} = z_{123} = z_{222} = z_{233} = z_{333} = 0\}\]
with coordinate $\lambda$ and the set $G \cdot Y$ comprising its $G$-translates is all of $V^o$.

However, a key difference between the subvariety $Y$ for this family and the subvariety $Y_L \subset M_{2 \times 4}(\C)$ for the Legendre family identified in the previous section is that, while $Y_L$ contained exactly one point in each $G$-orbit of $M_{2 \times 4}^o(\C)$, the intersection of $Y$ with a $G$-orbit in $V$ is not a single point but in fact (generically) six points, with an action of the symmetric group $S_3$. More precisely, if $y \in Y$ is a point with $\lambda = \lambda_0$, then
\[(G \cdot y) \cap Y = \left\{ \lambda = \lambda_0,\ 1 - \lambda_0,\ \frac{1}{\lambda_0},\ \frac{\lambda_0}{\lambda_0 - 1},\ \frac{1}{1 - \lambda_0},\ \frac{\lambda_0 - 1}{\lambda_0} \right\}\]
and the action of $S_3$ on this set is simply the action on the cross-ratio of a quadruple of (distinct) points in $\P^1$ familiar from enumerative geometry and the study of moduli problems (see, for example, \cite{kock-quantum-cohomology}). It can also be observed in the family of double covers of $\P^1$ as follows. The cross-ratio of a quadruple of points in $\P^1$ represented by a matrix in $M_{2 \times 4}^o$ is the bottom right entry of its orbit's representative in $Y_L$. Consider a point $\begin{bsmallmatrix}
    1 & 0 & 1 & 1 \\
    0 & 1 & 1 & \lambda_0
\end{bsmallmatrix} \in Y_L$ and permute the first three columns by an element of $S_3$; then the representative in $Y_L$ of the orbit of this column-permuted matrix is precisely the corresponding element in the above set.

Now we will show that $Y$ has the Galois restriction property, that is, that $\res_{V \to Y} : \C(V)^G \to \C(Y)$ is a Galois extension. Under our choice of normalization for the Aronhold invariants, the restrictions of $S$ and $T$ to $Y$ are 
\[\lambda^2 - \lambda + 1 \quad \text{and} \quad 2 \lambda^3 - 3 \lambda^2 - 3 \lambda + 2\]
respectively, thus
\[\res_{V \to Y}(J) = \frac{(\lambda^2 - \lambda + 1)^3}{-27 \lambda^2 (\lambda-1)^2}.\]
Clearing denominators and abusing notation by denoting $\res_{V \to Y} J$ simply as $J$, we have the equation
\[\lambda^6 - 3 \lambda^5 + (27J + 6) \lambda^4 + (-54J - 7) \lambda^3 + (27J + 6) \lambda^2 - 3 \lambda + 1 = 0\]
showing that $\C(Y) / \C(V)^G$ is a degree-six extension. It is easy to check that the expression $\frac{(\lambda^2 - \lambda + 1)^3}{-27 \lambda^2 (\lambda-1)^2}$ is invariant under the $W = S_3$-action on $Y$, therefore $\C(V)^G \subseteq \C(Y)^W$. Since $\C(Y) / \C(Y)^W$ is automatically Galois, as the base is the fixed field of a group action, and because the degree of $\C(Y) / \C(V)^G$ and the order of the group $W$ are equal, we conclude that $\C(V)^G = \C(Y)^W$. In other words, $\C(Y) / \C(V)^G$ is Galois with Galois group $W = S_3$, and so $(Y, W)$ has the Galois restriction property.

\begin{lemma} \label{lem:lambda-radical-expression}
    The sextic minimal polynomial of $\lambda$ over $\C(J)$ is solvable in radicals as
    \[\lambda = \frac{1 \pm \sqrt{1 + 4 \sqrt[3]{\frac{729J (54 J^2 + 18 J + 1) \pm 729 J \sqrt{4J + 1}}{2}}}}{2}.\]
\end{lemma}
\begin{proof}
    In fact, this sextic factors as the composition of a quadratic and a cubic. To see this, set $\alpha = \lambda (\lambda - 1)$, so that $J = \frac{(\alpha + 1)^3}{-27 \alpha^2}$. This rearranges to a cubic in $\alpha$ over $\C(J)$ whose solution is given explicitly as
    \[\alpha = \sqrt[3]{\frac{729J (54 J^2 + 18 J + 1) \pm 729 J \sqrt{4J + 1}}{2}}\]
    via the cubic formula. The equation $\alpha = \lambda (\lambda - 1)$ is a quadratic in $\lambda$ with explicit solution $\lambda = \frac{1 \pm \sqrt{1 + 4\alpha}}{2}$, proving the claim.
\end{proof}
The preceding lemma exhibits $\lambda$ as a multivalued function of $J$ on $Y$. From the point of view of the analytic mapping lemma, Lemma \ref{lem:analytic-mapping}, a generic $G$-orbit in $V^o$ intersects $Y$ in six points, each corresponding to a choice of a branch of this multivalued function. Fixing a choice of a branch in a neighborhood of $p \in V^o$ defines the map $\rho$ of the analytic mapping lemma.

Under the canonical Poincar\'e residue techinque, the periods of cubic hypersurfaces parametrized by $Y$ exactly coincide with the periods of the double covers in the Legendre family parametrized by $Y_L$, as the following lemma explains. Hence, we will use the Galois restriction property to lift the explicit solution $\varphi(\lambda)$ of the reduced Aomoto-Gelfand system $E'(2, 4)$ to the entire space $V$.

\begin{lemma} \label{lem:periods-legendre-form-hypersurfaces-p2}
	The period integrals of cubic hypersurfaces parametrized by $Y$ are exactly the period integrals of double covers of $\P^1$ parametrized by the subvariety $Y_L$ described in Section \ref{sec:elliptic-curves-double-covers}.
\end{lemma}
\begin{proof}
	The period integral of a cubic hypersurface parametrized by $Y$ is
	\[\int_\gamma \frac{x_1 dx_2 \wedge dx_3 - x_2 dx_1 \wedge dx_3 + x_3 dx_1 \wedge dx_2}{x_1^3 + (-\lambda - 1) x_1^2 x_3 + \lambda x_1 x_3^2 - x_2^2 x_3}\]
	where $\lambda$ is the coordinate on $Y$. Following the recipe of Lian and Yau \cite{Lian_2012}, the integrand is a meromorphic form on $\P^2$ with a pole along the cubic curve $E_\lambda$ defined by $x_1^3 + (-\lambda - 1) x_1^2 x_3 + \lambda x_1 x_3^2 - x_2^2 x_3 = 0$, and the expression for the period is meant to represent the integral of the Poincar\'e residue of this form over a $1$-cycle $\gamma$ in $E_\lambda$. In the chart $U = \{x_3 \neq 0\} \cong \C^2$ of $\P^2$, $E_\lambda$ is the curve $x_1^3 + (-\lambda - 1) x_1^2 + \lambda x_1 - x_2^2 = 0$ and the integrand becomes 
	\[\omega = \frac{dx_1 \wedge dx_2}{x_1^3 + (-\lambda - 1) x_1^2 + \lambda x_1 - x_2^2}.\]
	Writing $f(x_1,x_2) = x_1^3 + (-\lambda - 1) x_1^2 + \lambda x_1 - x_2^2$, the Poincar\'e residue of $\omega$ is the restriction to $E_\lambda$ of the $1$-form $\rho = a(x_1,x_2) dx_1 + b(x_1,x_2) dx_2$ on $U$ which satisfies $\omega = \frac{df}{f} \wedge \rho$. It is straightforward to compute
	\[\frac{df}{f} \wedge \rho = \frac{-2x_2 a(x_1,x_2) - (3x_1^2 + 2 (-\lambda - 1) x_1 + \lambda) b(x_1,x_2)}{f} dx_1 \wedge dx_2,\]
	and therefore the equation $\frac{df}{f} \wedge \rho = \frac{dx_1 \wedge dx_2}{f}$ implies that 
    \[a(x_1,x_2) = \frac{1}{-2x_2} - \frac{3x_1^2 + 2 (-\lambda - 1) x + \lambda}{2x_2} b(x_1,x_2).\]
    Now, pulling back $\rho$ along the closed embedding $\iota : E_\lambda \to U$, $\iota(x_1) = (x_1, \sqrt{x_1^3 + (-\lambda - 1) x_1^2 + \lambda x_2})$, we get the formula
    \begin{align*}
        \iota^* \rho &= \left( a ( x_1, \sqrt{x_1^3 + (-\lambda - 1) x_1^2 + \lambda x_1} ) + \frac{3x_1^2 + 2 (-\lambda - 1) x_1 + \lambda}{2x_2} b ( x_1, \sqrt{x_1^3 + (-\lambda - 1) x_1^2 + \lambda x_1} ) \right) dx_1 \\
        &= \frac{1}{-2 \sqrt{x_1^3 + (-\lambda - 1) x_1^2 + \lambda x_1}} dx_1 = -\frac{1}{2} \frac{dx_1}{\sqrt{x_1 (x_1-1) (x_1 - \lambda)}}.
    \end{align*}
	Up to scalar, this is exactly the integrand of the period $\int_\gamma \frac{x_1 dx_2 - x_2 dx_1}{\sqrt{x_1 x_2 (x_1 - x_2) (x_1 - \lambda x_2)}}$ of the double cover of $\P^1$ corresponding to $\begin{bsmallmatrix}
		1 & 0 & 1 & 1 \\
		0 & 1 & 1 & \lambda
	\end{bsmallmatrix} \in Y_L$, expressed in the chart $x_2 \neq 0$ of $\P^1$.
\end{proof}

\begin{theorem}[Period formula for elliptic curves in $\P^2$] \label{thm:periods-hypersurfaces}
    The period of the family of elliptic curves in $\P^2$ is given on an analytic open set in $V^o$ in terms of $\SL_3(\C)$-invariant functions by the power series
    \begin{align*}
        \Pi(Z) &= \sqrt[12]{\frac{729 (T^4 + 10 S^3 T^2 - 2 S^6) \pm 729  T \sqrt{(T^2 - 4 S^3)^3}}{2}} \\
        &\quad \times \sum_{n=0}^\infty \frac{\Gamma\left( n + \frac{1}{2} \right)^2}{\Gamma\left( \frac{1}{2} \right)^2 \Gamma(n + 1)^2} \left( \frac{1 \pm \sqrt{1 + 4 \sqrt[3]{\frac{729J (54 J^2 + 18 J + 1) \pm 729 J \sqrt{4J + 1}}{2}}}}{2} \right)^n.
    \end{align*}
    This function extends to a multivalued meromorphic function on $V^o$ by analytic continuation.
\end{theorem}
\begin{proof}
    The proof will follow the same lines as that of Theorem \ref{thm:periods-double-covers}. First, we combine Lemmas \ref{lem:analytic-lifting}, \ref{lem:lambda-radical-expression}, and \ref{lem:periods-legendre-form-hypersurfaces-p2} to construct a $G$-invariant function on $V^o$. Let $U$ be the Zariski open set in $V$ defined in Lemma \ref{lem:analytic-mapping}. By Lemma \ref{lem:periods-legendre-form-hypersurfaces-p2}, we know that the period formula should agree with the known solution $\varphi$ of $E'(2, 4)$ on $Y$. In an analytic neighborhood $U_p$ of any $p \in U$, the map $\rho : U_p \to Y$ is defined once a branch of the radical function in Lemma \ref{lem:lambda-radical-expression} is chosen. Then by the analytic mapping lemma, Lemma \ref{lem:analytic-lifting}, $\varphi(\rho(Z))$ is an analytic function on $U_p$ and constant on intersections of $G$-orbits with $U_p$. As a special case, we can take $p$ to lie in $Y$, and by choosing the branch that acts as the identity on $p$, the domain $U_p$ of the map $\rho$ intersects $Y$, and $\rho$ acts as the identity on $U_p \cap Y$. Hence, in this case, $\varphi(\rho(Z))$ additionally restricts to $\varphi$ on $U_p \cap Y$.
    
    Next, we seek a prefactor which is $\SL_3(\C)$-invariant, homogeneous of degree $-3$ with respect to the diagonal copy of $\C^\times$ in $\GL_3(\C)$ -- equivalently, homogeneous of degree $-1$ in the coordinates $z_{ijk}$ of $V$ -- and constant on $Y$. We will construct it by taking advantage of the algebraic relations between the restrictions to $Y$ of the Aronhold invariants $S$ and $T$. Namely, from the above forms of the restrictions of $S$ and $T$ to $Y$, it is easy to check that the equation
    \[(T^2 - 4 S^3) + 27 S^2 - 54 S + 27 = 0\]
    holds on $Y$. Therefore, the prefactor $P$ must satisfy
    \[(T^2 - 4 S^3) P^{12} + 27 S^2 P^8 - 54 S P^4 + 27 = 0\]
    on $Y$. This is a $G$-invariant polynomial by degree considerations, hence it holds on all of $V$. Computing by the cubic formula, we obtain the expression
    \[P(Z) = \sqrt[12]{\frac{729 (T^4 + 10 S^3 T^2 - 2 S^6) \pm 729  T \sqrt{(T^2 - 4 S^3)^3}}{2}}\]
    for our desired prefactor. The result is that the function
    \[\Pi(Z) = P(Z) \varphi(\rho(Z))\]
    agrees with $\varphi$ on a neighborhood in $Y$ and is annihilated by the symmetry operators of the tautological system. The period of this CY family also has these properties by Lemma \ref{lem:periods-legendre-form-hypersurfaces-p2}, and therefore the function $\Pi$ is exactly the holomorphic period of this CY family by Lemma \ref{lem:uniqueness-lemma}.
\end{proof}

We end by remarking on a subtlety in this period formula. In the proof of Theorem \ref{thm:periods-hypersurfaces}, we began with the solution $\varphi(\lambda)$ on $Y$, defined in a neighborhood in $Y$ by the hypergeometric function $\leftindex_2{F}_1 \left( \frac{1}{2}, \frac{1}{2}; 1; \lambda \right)$ and extended to a multivalued function, which we also denote $\varphi(\lambda)$, on the rest of $Y$ by analytic continuation. Then we lifted $\varphi(\lambda)$ to a function $P(Z) \varphi(\rho(Z))$ on a neighborhood in $V^o$ where the second factor is constant on $G$-orbits, which extends to a multivalued function $\Pi(Z)$ on all of $V^o$ by analytic continuation. Each $G$-orbit in $V^o$ intersects $Y$ at multiple points $p_i$, and $P(Z)$ is constant on $Y$, so by $\SL_3(\C)$-invariance of $\varphi(\rho(Z))$ it appears that $\Pi(Z)$ should take the same value at each $p_i$ thanks to the analytic lifting lemma, Lemma \ref{lem:analytic-lifting}. But the $p_i$ are transformed into one another by elements of $G = \GL_3(\C)$ with nontrivial determinant, so the transformation properties of the tautological system suggest that $\Pi(Z)$ should take different values at the $p_i$. At first glance, this seems contradictory.

Recall, however, that the space of solutions to the Gauss hypergeometric differential equation is two-dimensional, in this case spanned by $\leftindex_2{F}_1 \left( \frac{1}{2}, \frac{1}{2}; 1; \lambda \right)$ and a modification thereof involving $\log \lambda$ and a power series in $\lambda$ whose coefficients involve the digamma function \cite{NIST:DLMF}. It turns out that the transformation properties of $\Pi(Z)$ encoded in the tautological system result from the relationship between the different branches of $\varphi(\lambda)$; in other words, the transformation properties of $\Pi(Z)$ under the $\GL_3(\C)$-action depend on whether one analytically continues $\Pi$ along $Y$ or along a $G$-orbit.

This discrepancy reflects a tension in the current form of the theory: from the standpoint of tautological systems, the periods should be considered as sheafy multivalued objects, a perspective which is not realized in the reduction procedure of $E(n, 2n)$. This highlights the pressing need to understand the reduction of $\D$-modules to subvarieties in order to adapt the method to more general contexts. We anticipate that invariant theory, and restriction properties in particular, will play a vital role in these questions.

\section*{Acknowledgments}
We thank Gerry Schwarz and the anonymous referee for their helpful comments on earlier versions of the manuscript. We also extend deep gratitude to Gerry Schwarz for insightful conversations at the early stages of this work which greatly influenced the direction of our research. K.S. thanks the Shanghai Institute for Mathematics and Interdisciplinary Sciences for their generous hospitality during which part of this work was completed.

\printbibliography

\end{document}